\documentclass[a4paper,10pt]{article}

\usepackage[left=1in,right=1in,top=1in,bottom=1.2in]{geometry}
\usepackage{authblk} 
\usepackage[titletoc]{appendix} 
\usepackage[T1]{fontenc} 
\usepackage{lmodern} 
\usepackage{mathrsfs} 
\usepackage{xcolor} 
\usepackage{bm} 
\usepackage{amsmath,amsfonts,amssymb,amsthm} 
\usepackage{thmtools}
\usepackage{mathtools}
\usepackage{graphicx,float} 
\usepackage{caption,subcaption} 
\usepackage[colorlinks=true,linkcolor=blue,citecolor=blue,urlcolor=blue]{hyperref} 
\usepackage{url} 
\usepackage[sort,authoryear,round]{natbib} 
\usepackage[capitalize,nameinlink,compress]{cleveref} 
\usepackage{tikz-cd} 
\usepackage{amscd} 
\usepackage[linesnumbered,ruled,vlined]{algorithm2e} 
\usepackage{tabularx} 
\usepackage{nicematrix} 
\usepackage{enumitem} 
\setlist[enumerate,1]{label={\upshape(\roman*)}} 
\usepackage{centernot}

\newtheorem{theorem}{Theorem}[section]
\newtheorem*{theorem*}{Theorem}
\newtheorem{proposition}[theorem]{Proposition}
\newtheorem{lemma}[theorem]{Lemma}

\newtheorem{definition}[theorem]{Definition}
\newtheorem{remark}[theorem]{Remark}
\newtheorem{claim}[theorem]{Claim}
\numberwithin{equation}{section}

\crefname{figure}{Figure}{Figures} 
\crefname{equation}{}{} 
\AtBeginEnvironment{appendices}{\crefalias{section}{appendix}}

\newcommand{\calm}{\mathcal{M}}
\newcommand{\calx}{\mathcal{X}}

\newcommand{\calc}{\mathcal{C}}

\newcommand{\calf}{\mathcal{F}}
\newcommand{\calv}{\mathcal{V}}
\newcommand{\calg}{\mathcal{G}}
\newcommand{\caln}{\mathcal{N}}

\newcommand{\dm}{d_{\mathcal{M}}}
\newcommand{\dmu}{\,\mathrm{d}\mu}
\newcommand{\diam}{\mathrm{diam}}

\newcommand{\bbe}{\mathbb{E}}
\newcommand{\bbr}{\mathbb{R}}
\newcommand{\bbp}{\mathbb{P}}
\newcommand{\bbs}{\mathbb{S}}
\newcommand{\bbz}{\mathbb{Z}}

\newcommand{\bbn}{\mathbb{N}}

\newcommand{\rmr}{\mathrm{R}}
\newcommand{\sqn}{\sqrt{n}}
\newcommand{\grad}{\mathrm{grad}}
\newcommand{\hess}{\mathrm{Hess}}
\newcommand{\diff}[1]{\,\mathrm{d}#1}
\newcommand{\cost}{\mathbf{K}}
\newcommand{\gcost}{\mathbf{G}}

\newcommand{\bms}{\mathbf{S}}
\newcommand{\bmf}{\mathbf{F}}
\newcommand{\vol}{\mathrm{vol}}
\newcommand{\rlog}[2]{\mathrm{Log}_{#1}(#2)}
\newcommand{\rexp}{\mathrm{Exp}}
\newcommand{\hk}{K(t,x,y)}
\newcommand{\pto}{\xrightarrow{\enskip \mathbb{P}\enskip}}
\newcommand{\asto}{\xrightarrow{\enskip a.s.\enskip}}

\newcommand{\wto}{\xrightarrow{\enskip w\enskip}}

\makeatletter
\def\@title{
  \begin{center}
    \fontsize{12}{14.4}\selectfont\bfseries\MakeUppercase \@titletext
  \end{center}
}
\newcommand{\titletext}[1]{\gdef\@titletext{#1}}
\makeatother

\titletext{Uniform Limit Theorems for Diffusion Means on Compact Riemannian Manifolds}
\author{Yueqi Cao} 
\affil{Department of Mathematics, KTH Royal Institute of Technology}
\date{}

\begin{document}

\maketitle

\begin{abstract}
Motivated by Varadhan's theorem, we study asymptotic properties of the diffusion Fr\'echet functions, variances, and means introduced in \citet{eltzner2023diffusion}, on compact Riemannian manifolds. Given independent and identically distributed samples, we prove uniform laws of large numbers and uniform central limit theorems for their empirical counterparts. Through the small time asymptotics of gradients and Hessians of diffusion Fr\'echet functions, we build a strong connection to the central limit theorem for classical Fr\'echet means on compact Riemannian manifolds by recovering the ``non-standard'' correction term from \cite{hotz2024central}.

\end{abstract}


\section{Introduction}

The Fr\'echet mean is a natural generalization of the standard mean in Euclidean spaces to arbitrary metric spaces \citep{frechet1948elements}. Suppose $\calm$ is a metric space with distance function $\dm$ and probability distribution $\mu$. The Fr\'echet mean is defined as any minimizer of the Fr\'echet function
\begin{equation}\label{eq:frechet-F}
F_\mu(x)=\int_\calm\dm^2(x,\xi)\diff{\mu}(\xi).
\end{equation} 
When the metric space $\calm$ is specialized to a Riemannian manifold, the Riemannian structure further endows the Fr\'echet mean with rich geometric and probabilistic properties. 
In particular, the existence, uniqueness, and convexity properties of Fr\'echet means on Riemannian manifolds have been studied in  \cite{kendall1990probability,karcher1977riemannian,afsari2011riemannian}.

In statistics, a central problem is to understand the asymptotic behavior of the empirical Fr\'echet means. On general metric spaces, the strong consistency of empirical Fr\'echet means was studied in \cite{ziezold1977lawoflargenumber}. A large sample theory on manifolds was developed by \cite{bhattacharya2003large, bhattacharya2005large}. There is an extensive literature on central limit theorems (CLTs) for Fr\'echet means on Riemannian manifolds, including  \cite{bhattacharya2008statistics,kendall2011limit,bhattacharya2017omnibus,eltzner2021stability}. However, most existing CLTs require either that the population distribution is supported away from the cut locus, or that the cut locus has a high codimension relative to $\calm$. Due to the existence of cut locus,  the Fr\'echet function $F_\mu(x)$ in \eqref{eq:frechet-F} is not always twice differentiable \citep{lytchak2025zeroprobabilitycutlocus}. Recently  \cite{hotz2024central} proved a global CLT for Fr\'echet means allowing the support of the distribution to cover the cut locus. Their geometric analysis first revealed a ``non-standard'' correction term in the Hessian of the population Fr\'echet function when the cut locus has codimension one.

Since the main difficulty of establishing a CLT for Fr\'echet means lies in the nonsmoothness of the squared Riemannian distance function and the Fr\'echet function, a natural question is whether these functions can be approximated by other tractable, smooth functions intrinsically defined on the manifold. The celebrated Varadhan's theorem offers such candidates. 
\begin{theorem*}[\citep{varadhan1967behavior}]
    Let $K(t,x,y)$ be the heat kernel on a complete Riemannian manifold $\calm$. Then for any $x,y\in\calm$, 
    $$
    \lim_{t\to 0^+}-2t\log(K(t,x,y))= \dm^2(x,y),
    $$
    uniformly on any compact sets of $\calm\times\calm$.
\end{theorem*}

If we take $-2t\log(K(t,x,y))$ as the integrand and define   analogously a Fr\'echet-type function, its minimizers serve as smooth surrogates of the Fr\'echet means. This approach was taken in \cite{eltzner2023diffusion}, where the minimizers are called \emph{diffusion means}.

This paper takes Varadhan’s theorem as the starting point for studying asymptotic behaviors of diffusion means \emph{uniformly in time}.  Our results generalize many existing results on diffusion means and Fr\'echet means, and provide a unified viewpoint for studying their statistical properties.

\subsection{Overview of the Main Results}
Suppose $\calm$ is a compact Riemannian manifold with a probability distribution $\mu$. In this paper, we consider the following function on $\calm$
\begin{equation*}
    F^t_\mu(x) = \int_\calm -2t\log(K(t,x,\xi))\diff{\mu}(\xi).
\end{equation*}
We call this the diffusion $t$-Fr\'echet function to highlight its parallel with the classical Fr\'echet function. The corresponding minimum $V^t_\star = \min_{x\in\calm} F^t_\mu(x)$ is called the diffusion $t$-Fr\'echet variance. Any minimizer
\begin{equation*}
x^t_\star \in \mathop{\arg\min}_{x\in\calm} F^t_\mu(x)
\end{equation*}
is known as a \emph{diffusion $t$-mean} \citep{eltzner2023diffusion}. Throughout the paper we always assume that the (population) diffusion $t$-mean is unique. Suppose $\xi_1,\ldots,\xi_n$ are i.i.d. samples. We present asymptotic analysis for the empirical function $F^t_n$, variance $V^t_n$, and mean $x^t_n$. The results are summarized in the following.
\begin{itemize}
    \item Uniform laws of large numbers (ULLNs) for diffusion $t$-Fr\'echet functions and variances, and means (\Cref{thm:ulln-vf-function} and \Cref{thm:ulln-vf-mean-var}). Roughly speaking, we show that as $n\to\infty$,
    \begin{equation*}
        \begin{aligned}
            & \sup_{t\in[0,1]}\sup_{x\in\calm}|F^t_n(x)-F^t_\mu(x)|\asto0,\\
            &  \sup_{t\in[0,1]}|V^t_n- V^t_\star|\asto 0,\\
            &\sup_{t\in[0,1]}\dm(x^t_n, x^t_\star)\asto 0;
        \end{aligned}
    \end{equation*}
    \item Uniform central limit theorems (UCLTs) for diffusion $t$-Fr\'echet functions and variances (\Cref{thm:clt-vf-function} and \Cref{thm:clt-vf-var}). Roughly speaking, we show that as $n\to\infty$,
    \begin{equation*}
        \begin{aligned}
            & \sqn(F^t_n(x)-F^t_\mu(x))\wto \text{Gaussian process indexed by }[0,1]\times \calm,\\
            &  \sqn(V^t_n-V^t_\star)\wto \text{Gaussian process indexed by } [0,1].
        \end{aligned}
    \end{equation*}
    \item 
    Fixed-$t$ CLT for diffusion means for each $t\in (0,1]$ (\Cref{thm:clt-vf-mean}). Assuming the Hessian of $F^t_\mu$ is positive definite at $x^t_\star$, as $n\to\infty$, 
    \begin{equation*}
    \sqn\,\rlog{x^t_\star}{x^t_n}\wto \mathcal{N}(0,\Sigma^t).
    \end{equation*} 
     We notice that the covariance matrix $\Sigma^t$ is closely related to the gradient and Hessian of the diffusion $t$-Fr\'echet function. Assuming that $\mu$ is absolutely continuous with respect to the volume measure on $\calm$, we show that
\begin{align}
    & \lim_{t\to 0^+}\grad_x(F^t_\mu) = \bbe_{\xi\sim\mu}[\grad_x(\dm^2(\cdot,\xi))], \tag{\Cref{prop:gradient-converge}}\\
    & \lim_{t\to 0^+}\hess_x(F^t_\mu)  = \bbe_{\xi\sim\mu}[\hess_x(\dm^2(\cdot,\xi))] + J_\mu(x). \tag{\Cref{prop:hessian-converge}}
\end{align}
If, in addition, there is a compact neighborhood $U$ of the Fr\'echet mean $x^0_\star$ such that the Hessians $\hess_{x}(F^t_\mu)$ converge uniformly, then the Fr\'echet function $F^0_\mu$ is twice differentiable on $U$, and we prove a UCLT for diffusion means over a time interval that includes zero (\Cref{thm:clt-fre-mean}). Roughly speaking, if $\hess_{x^0_\star}(F^0_\mu)$ is positive definite, there exists $t_0\in (0,1]$ such that, as $n\to\infty$,
\begin{equation*}
    \{\sqn\,\rlog{x^t_\star}{x^t_n}\}_{t\in [0,t_0]}\wto \bbr^m\text{-valued Gaussian process indexed by }[0,t_0].
    \end{equation*}
As a corollary we obtain a CLT for the Fr\'echet mean by evaluating at $t=0$, and recover the same correction term $J_\mu(x)$ as in \cite{hotz2024central}. The Varadhan approach, however, requires no assumptions on the geometry of the cut locus.
\end{itemize}

\subsection{Related Works}

Using the minimizer of the expected negative logarithmic heat kernel as a smooth approximation of the classical Fr\'echet mean was first introduced by \cite{eltzner2023diffusion}, where they coined the term \emph{diffusion mean}.  They also noted a connection to Varadhan's theorem in \cite[Section 4.1]{eltzner2023diffusion}. From a general standpoint, they defined diffusion $t$-mean as the minimizer of expected negative log-likelihood of Brownian motion $x\mapsto \bbe[-\log(K(t,x,\Xi))]$. The factor $2t$ does not affect the definition of diffusion means, but it is crucial for deriving the asymptotics of the objective function and of its minimum relative to their Fr\'echet counterparts. Indeed, without it, the rate estimates for the gradients and Hessians of the logarithmic heat kernel would be lost.

For fixed $t>0$, an LLN and a CLT for diffusion $t$-means were proved  \citep[Lemma 4.3, Theorem 4.5]{eltzner2023diffusion}. By contrast, the LLNs and CLTs developed in this paper hold uniformly over the time interval and are more general in that they cover the function, variance, and mean simultaneously.

Our work is closely related to the recent work by \cite{hotz2024central}, where the authors proved a general CLT for Fr\'echet means which allows the population distribution to have support containing the cut locus. Through a careful geometric analysis of the cut locus, they first identified the ``non-standard'' correction term in the Hessian of the population Fr\'echet function that arises when the cut locus has codimension one. Our UCLT of diffusion means establishes a strong connection to their results by recovering the same correction term. We provide an example in \Cref{sec:counter-example} in which the density assumption in \cite{hotz2024central} fails but the Fr\'echet mean still satisfies the CLT as a consequence of \Cref{thm:clt-fre-mean}.  \\

\noindent\textbf{Other Generalized Fr\'echet-Type Statistics.} There is a rich literature extending the classical notion of Fr\'echet means, which provide a foundational framework for statistical inference in non-Euclidean geometries \citep{huckemann2026inference,yang2010riemannian,arnaudon2012medians}. Recent advances generalize these constructions by replacing squared geodesic distances with arbitrary cost functions or more general metric structures, leading to the study of generalized Fr\'echet functions, variances, and means \citep{schotz2019convergence, ramos2010one, brunel2024concentration}. On Riemannian manifolds, these frameworks produce robust estimators like the geometric median as well as $L^p$-centers of mass \citep{afsari2011riemannian}. Another line of recent work approximates the Fr\'echet function on Riemannian manifolds smoothly via the Huber loss function, where the corresponding minimizers are called Huber means \citep{lee2026huber}. Diffusion means also fall into the category of generalized Fr\'echet means. In this paper, we focus on the theoretical asymptotic properties of diffusion means and do not compare them with other types of generalized Fr\'echet means.

\section{Setup}

Let $(\calm,g)$ be a compact connected Riemannian manifold without boundary of dimension $m$. Let $\Delta_g$ be the Laplace--Beltrami operator. The heat equation on $\calm$ is given by
\begin{equation*}
\begin{dcases}
\partial_tu(t,x) = \frac{1}{2}\Delta_g u(t,x), & (t,x)\in (0,+\infty)\times\calm, \\
\lim_{t \to 0^+}u(t,x) = f(x), & x\in \calm.
\end{dcases}
\end{equation*}
Let $\vol_\calm$ be the volume measure. The heat kernel $\hk$ is a smooth function on $(0,+\infty)\times\calm\times\calm$ which solves the heat equation by setting
\begin{equation*}
u(t,x) = \int_\calm \hk f(y)\diff{\vol_\calm}(y).
\end{equation*}
Define a parameterized family of functions on $\calm\times\calm$ by
\begin{equation*}
    \cost^t(x,y) = \begin{dcases}
        -2t\log(\hk), & t>0,\\
        \dm^2(x,y), & t=0,
    \end{dcases}
\end{equation*}
where $\dm(\cdot,\cdot)$ is the Riemannian distance function. Varadhan's theorem claims that $\cost^t(x,y)\to \cost^0(x,y)$ uniformly on $\calm\times\calm$ as $t\to 0^+$.

\begin{definition}

Let $\Xi$ be an $\calm$-valued random variable whose associated probability distribution is $\mu$. For any $t\geq 0$, the diffusion $t$-Fr\'echet function is defined as 
\begin{equation}\label{eq:v-func}
    F^t_\mu(x) = \bbe[\cost^t(x,\Xi)].
\end{equation}
The minimal value of \eqref{eq:v-func} is called the diffusion $t$-Fr\'echet variance
\begin{equation}\label{eq:v-var}
    V^t_\star = \min_{x\in\calm} F^t_\mu(x).
\end{equation}
A diffusion $t$-(Fr\'echet) mean \citep{eltzner2023diffusion} is any minimizer of \eqref{eq:v-func}
\begin{equation}\label{eq:v-mean}
    x_\star^t \in \mathop{\arg\min}_{x\in\calm} F^t_\mu(x).
\end{equation}
\end{definition}
For $t=0$, \eqref{eq:v-func}, \eqref{eq:v-var} and \eqref{eq:v-mean} are exactly the definitions for classical Fr\'echet function, variance and mean.



Suppose $\xi_1,\ldots,\xi_n\in\calm$ are i.i.d. samples of $\Xi$. The corresponding empirical versions of \eqref{eq:v-func}, \eqref{eq:v-var} and \eqref{eq:v-mean} are given by 
\begin{equation*}
    F^t_n(x) = \frac{1}{n}\sum_{i=1}^n\cost^t(x,\xi_i),\quad V^t_n = \min_{x\in\calm} F^t_n(x),\quad x^t_n\in\mathop{\arg\min}_{x\in\calm} F^t_n(x). 
\end{equation*}

Throughout the paper, we always assume that the population diffusion $t$-mean is unique for each $t\geq 0$. The empirical diffusion $t$-means are not necessarily unique. We will assume that for each $t\geq 0$ and $n\in\bbn$, a measurable selection $x^t_n$ from the set of minimizers of $F^t_n$ is made, so that $x^t_n$ is a well-defined $\calm$-valued random variable. The existence of a measurable selection is guaranteed by the measurable selection theorem \citep[Theorem 7.34]{shapiro2021lectures}.

\section{Uniform Laws of Large Numbers}\label{sec:consistency}

First of all, we show that Varadhan's theorem implies convergence of diffusion $t$-Fr\'echet functions and variances, and diffusion means to their Fr\'echet counterparts.

\begin{proposition}\label{prop:int-vf}
    Let $\calm$ be a compact Riemannian manifold. Then
    \begin{enumerate}
        \item $\|F^t_\mu- F^0_\mu\|_\infty\to 0$ as $t\to 0^+$.
        \item $V^t_\star\to V^0_\star$ as $t\to 0^+$.
        \item Assuming uniqueness of diffusion $t$-means, then $\dm(x^t_\star, x^0_\star)\to 0$ as $t\to 0^+$.
    \end{enumerate}
\end{proposition}

\begin{remark}
    Without assumption of uniqueness, it was proved that the sets of diffusion $t$-means converge to the set of Fr\'echet means in the sense of Ziezold \citep[Theorem 4.7]{eltzner2023diffusion}. Here we focus on the case of unique means for simplicity. 
\end{remark}

By \Cref{prop:int-vf}, the function $(t,x)\mapsto F^t_\mu(x)$ is continuous at each point in $\{0\}\times\calm$. Since the function is smooth on $(0,1]\times\calm$, it is uniformly continuous over the compact domain $[0,1]\times \calm$. Similarly, \Cref{prop:int-vf} also implies the uniform continuity of the function $t\mapsto V^t_\star$ and the map $t\mapsto x^t_\star$ over $[0,1]$.

For each fixed $t\geq 0$ and $x\in\calm$, by the classical law of large numbers (LLN), we have $F^t_n(x)\to F^t_\mu(x)$ almost surely as $n\to\infty$. The following result establishes uniform convergence over $[0,1]\times \calm$, giving a uniform law of large numbers (ULLN) for diffusion $t$-Fr\'echet functions.

\begin{theorem}\label{thm:ulln-vf-function}
    For any probability distribution $\mu$ on $\calm$,
    $$
    \sup_{t\in[0,1]}\sup_{x\in\calm}|F^t_n(x)-F^t_\mu(x)|\asto0,\,\text{ as }n\to\infty.
    $$
\end{theorem}


As a consequence of \Cref{thm:ulln-vf-function}, we can also prove ULLNs for diffusion $t$-Fr\'echet variances and means.

\begin{theorem}\label{thm:ulln-vf-mean-var}
For any probability distribution $\mu$ on $\calm$,
\begin{equation*}
    \sup_{t\in[0,1]}|V^t_n- V^t_\star|\asto 0, \text{ as }n\to\infty.
    \end{equation*}
 Assuming uniqueness of diffusion $t$-means, then
    \begin{equation*}
    \sup_{t\in[0,1]}\dm(x^t_n, x^t_\star)\asto 0,\text{ as }n\to \infty.
    \end{equation*}

\end{theorem}

\begin{remark}
 For fixed $t>0$, a strong LLN for diffusion $t$-means was proved in \citep[Lemma 4.2 and Lemma 4.3]{eltzner2023diffusion}. For $t=0$, \Cref{thm:ulln-vf-mean-var} implies a strong LLN for Fr\'echet means on compact Riemannian manifolds, i.e.,
\begin{equation*}
\dm(x^0_n,x^0_\star)\asto0,\,\text{ as }n\to\infty,
\end{equation*}
which follows from classical strong laws for Fr\'echet means \citep{ziezold1977lawoflargenumber}; see also \citep[Proposition 1]{hotz2024central}.
\end{remark}



\section{Uniform Central Limit Theorems for Diffusion Fr\'echet Functions and Variances}\label{sec:clt-func-var}

In this section our goal is to establish uniform central limit theorems (UCLTs) for diffusion $t$-Fr\'echet functions and variances. To start with, let us take the diffusion Fr\'echet variance as an example.  For fixed $t\in [0,1]$, the diffusion Fr\'echet variance $V^t_n$ is a random variable, and by the standard CLT, $\sqrt{n}(V^t_n-V^t_\star)$
converges to a centered Gaussian in distribution. With the time variable involved, the diffusion Fr\'echet variance $(V^t_n)_{t\in [0,1]}$, regarded as a function over $[0,1]$, is a random process. We will show that the random process $\{\sqrt{n}(V^t_n-V^t_\star)\}_{t\in [0,1]}$ will converge to a centered Gaussian process indexed by $[0,1]$.  

Given a class $\calf$ of $\mu$-integrable measurable functions over a metric space $\calx$, the empirical process indexed by $\calf$ is defined as 
$$
\nu_n(f) = \sqn\bigg(\int_\calx f\diff{\mu_n}-\int_\calx f\diff{\mu}\bigg)= \sqn(\mu_n-\mu)(f), \quad\forall f\in\calf.
$$
To show that the empirical process $\nu_n$ converges weakly to a Gaussian process, the function class $\calf$ should not be too ``large''. The standard way is to show that the function class is \emph{Donsker}. We include  background on Donsker classes in \Cref{app:donsker}.

We will consider the following function class on $\calm$:
$$
\calv = \{\bmf^t_x(\cdot)=\cost^t(x,\cdot):(t,x)\in[0,1]\times\calm\}\subseteq C(\calm).
$$
The key to showing that this function class is Donsker is a delicate control of the diffusion $t$-Fr\'echet functions (\Cref{lemma:temporal}):
$$
\|\bmf^t_x-\bmf^s_y\|_\infty\lesssim \big(\dm(x,y)+|t-s|^{1/3}\big).
$$
Hence the function class is Lipschitz in the spatial parameter and H\"older in the temporal parameter. Then it follows that the bracketing entropy of the function class $\calv$ is controlled by the metric entropy of the parameter space $[0,1]\times\calm$. This leads to the following UCLT of diffusion $t$-Fr\'echet functions.


\begin{theorem}\label{thm:clt-vf-function}
    For any probability distribution $\mu$ on $\calm$, as $n\to\infty$,
    $$
\bigg\{\sqn(F^t_n(x)-F^t_\mu(x))\bigg\}_{(t,x)\in[0,1]\times\calm}
\wto \{G^t_\mu(x)\}_{(t,x)\in[0,1]\times\calm},
$$
where $G_\mu$ is a centered Gaussian process with  covariance structure
\begin{equation}\label{eq:covar-gp}
\bbe[G^t_\mu(x)G^s_\mu(y)]
=\bbe[\bmf^t_x(\Xi)\bmf^s_y(\Xi)]
-F^t_\mu(x)F^s_\mu(y).
\end{equation}
\end{theorem}

The UCLT for diffusion $t$-Fr\'echet variances is obtained by evaluating the empirical process $\nu_n(\bmf^t_x)$ and the Gaussian process $G^t_\mu(x)$ at the continuous path $t\mapsto x^t_\star$. The technical step is to show that the difference between the two processes $\sqrt{n}(V^t_n-V^t_\star)-\nu_n(\bmf^t_{x^t_\star})$ converges to zero in probability uniformly over $t\in [0,1]$. We then obtain the following.


\begin{theorem}\label{thm:clt-vf-var}
For any probability distribution $\mu$ on $\calm$, as $n\to\infty$,
    $$
    \big\{\sqn(V^t_n-V^t_\star)\big\}_{t\in [0,1]}\wto \{G^t_\mu(x^t_\star)\}_{t\in[0,1]},
    $$
    where the covariance structure is induced by \eqref{eq:covar-gp}.
\end{theorem}

\begin{remark}
  At $t=0$, \Cref{thm:clt-vf-var} implies the CLT for Fr\'echet variance
  $$
  \sqrt{n}(V^0_n-V^0_\star)\wto \caln(0, \sigma^0),
  $$
  where  
  $$
  \sigma^0 = \bbe[\dm^4(x^0_\star,\Xi)]-(V^0_\star)^2,
  $$
  which is consistent with the CLT for Fr\'echet variances on compact Riemannian manifolds \citep[Theorem 1]{dubey2019frechet}.  
\end{remark}


\section{Uniform Central Limit Theorems for Diffusion Means}\label{sec:clt-mean}

\subsection{Fixed-$t$ Central Limit Theorem}

For any $x\in\calm$, its cut locus $\mathrm{Cut}_x$ consists of points beyond which geodesics starting at $x$ cease to be minimizing. Let $\rexp_x$ be the exponential map which sends each tangent vector $v\in T_x\calm$ to $\rexp_x(v)\in\calm$. Its inverse $\mathrm{Log}_x$, the Riemannian logarithm map, is well-defined on $\calm\setminus\mathrm{Cut}_x$. The injectivity radius of $\calm$ is defined as $i_\calm=\min_x d_\calm(x,\mathrm{Cut}_x)$. By \Cref{thm:ulln-vf-mean-var}, for large enough $n\in\bbn$, almost surely we have $\dm(x^t_\star,x^t_n)<i_\calm$, thus the tangent vectors $\rlog{x^t_\star}{x^t_n}\in T_{x^t_\star}\calm$ are almost surely well-defined. The CLT for diffusion $t$-means is formulated in terms of the tangent vectors $\rlog{x^t_\star}{x^t_n}$.

Let $\nabla$ denote the Levi-Civita connection. For any smooth function $f:\calm\to\bbr$, let $\grad_x(f)$ and $\hess_x(f)$ denote its gradient and Hessian at $x$, respectively. We view $\hess_x(f)$ as the self-adjoint operator $v\mapsto \nabla_v\grad(f)$ on $T_x\calm$. By fixing an orthonormal basis of $T_x\calm$, we identify $\grad_x(f)$ as an $m$-dimensional column vector and   $\hess_x(f)$ as an $m\times m$ matrix. We also denote $\bms^t_y(\cdot)=\cost^t(\cdot,y)$.

\begin{theorem}\label{thm:clt-vf-mean}
    Fix $t>0$. Assume that the Hessian $\hess_{x^t_\star}(F^t_\mu)$ is  positive definite. Then
    \begin{equation*}
    \sqn\,\rlog{x^t_\star}{x^t_n}\wto \mathcal{N}(0,\Sigma^t),\,\text{ as }n\to\infty,
    \end{equation*}
    where the covariance matrix $\Sigma^t$ is given by
    \begin{equation*}
     \Sigma^t=\big(\hess^{-1}_{x^t_\star}(F^t_\mu)\big)\bbe\big[(\grad_{x^t_\star}(\bms^t_\Xi))(\grad_{x^t_\star}(\bms^t_\Xi))^\top\big]\big((\hess^{-1}_{x^t_\star}(F^t_\mu)\big)^\top.
    \end{equation*}
\end{theorem}

\begin{remark}
    A smeary CLT for diffusion $t$-means was proved in \citep[Theorem 4.5]{eltzner2023diffusion}. 
    Our fixed-$t$ CLT is the nondegenerate quadratic case of that theorem. We state its covariance $\Sigma^t$ explicitly to connect it to the CLT for Fr\'echet means in \cite{hotz2024central}. 
\end{remark}

For each $t>0$, thanks to the smoothness of $\cost^t(x,y)$ and compactness of $\calm$, we can exchange integration with differentiation, hence
\begin{equation}\label{eq:exchange}
    \begin{aligned}
         \grad_x(F^t_\mu) = \bbe[ \grad_x(\bms^t_\Xi)],\quad
         \hess_x(F^t_\mu) = \bbe[ \hess_x(\bms^t_\Xi)].
    \end{aligned}
\end{equation}
In general \eqref{eq:exchange} does not hold at $t=0$ since $\cost^0(x,y)=\dm^2(x,y)$ is not smooth. The inherent regularity explains why establishing a CLT for diffusion means is simpler than for Fr\'echet means.


\subsection{Convergence of Gradients and Hessians}

Assume that the probability measure $\mu$ is absolutely continuous to the volume measure $\vol_\calm$. For each $x\in\calm$, its cut locus is a null set $\mu(\mathrm{Cut}_x)=0$ \citep[Proposition III.3.1.]{chavel2006riemannian}.
Since the squared distance $\dm^2(x,y)$ is Lipschitz in $x$ with uniformly bounded Lipschitz constant on $\calm$ and smooth at $x$ for $\mu$-almost every $y$, differentiation under the integral sign gives \citep[Theorem 2.13]{LeGall2022}
\begin{equation*}
    \grad_x(F^0_\mu) = \grad_x\bbe[\bms^0_\Xi]=\bbe[\grad_x(\bms^0_\Xi)].
\end{equation*}
Using the logarithmic heat kernel estimates in \cite{chen2023logarithmic}, we can show that the gradients of diffusion $t$-Fr\'echet functions converge pointwise to the gradient of the classical Fr\'echet function.

\begin{proposition}\label{prop:gradient-converge}
    Assume that $\mu\ll \vol_\calm$. Then for any $x\in\calm$, $\grad_x(F^t_\mu) \to \grad_x(F^0_\mu)$ as $t\to 0^+$.
\end{proposition}

For any $x\in\calm$ and $\delta>0$, let $B_\delta(x)=\{y\in\calm:\dm(x,y)<\delta\}$ be the open ball of radius $\delta$ at $x$. Define
\begin{equation*}
    \calc_\delta(x)  = \bigcup_{y\in\mathrm{Cut}_x} B_\delta(y) =\{z\in\calm:\dm(z,\mathrm{Cut}_x)<\delta\}.
\end{equation*}
Under the assumption that $\mu\ll \vol_\calm$, it follows that $\mu(\calc_\delta(x))\to 0$ as $\delta\to 0^+$. Furthermore, we define
\begin{equation}\label{eq:Jtdelta}
    J^{t,\delta}_\mu(x) = \int_{\calc_\delta(x)}\hess_x(\bms^t_\xi)\diff{\mu}(\xi).
\end{equation}
For a fixed $\delta>0$, the limit $\displaystyle\lim_{t\to 0^+}J^{t,\delta}_\mu(x)$ may not exist. The following proposition gives a criterion for the pointwise convergence of $\hess_x(F^t_\mu)$.
\begin{proposition}\label{prop:hessian-converge}
    Assume that $\mu\ll\vol_\calm$. For any $x\in\calm$, suppose the function $\xi\mapsto \hess_x(\bms^0_\xi)$ is $\mu$-integrable, then $\displaystyle\lim_{t\to 0^+}\hess_x(F^t_\mu)$ exists if and only if the following two limits are finite and equal
    \begin{equation}\label{eq:limitJ}
        \lim_{\delta\to 0^+}\varlimsup_{t\to 0^+} J^{t,\delta}_\mu(x)  = \lim_{\delta\to 0^+}\varliminf_{t\to 0^+} J^{t,\delta}_\mu(x).
    \end{equation}
\end{proposition}

Compared to the convergence of gradients in \Cref{prop:gradient-converge}, the convergence of Hessians in \Cref{prop:hessian-converge} does not imply that $\displaystyle\lim_{t\to 0^+}\hess_x(F^t_\mu)=\hess_x(F^0_\mu)$. The reasons are twofold: (1) at $t=0$, the Fr\'echet function is not always twice differentiable, hence $\hess_x(F^0_\mu)$ may not exist (see, e.g., \cite[Theorem 7.3]{lytchak2025zeroprobabilitycutlocus}); (2) even if $\hess_x(F^0_\mu)$ exists, we lack control of $\hess_x(F^t_\mu)$ near the cut locus. While for gradients $\grad_x(F^t_\mu)$ there is a uniform bound for $t\in[0,1]$, for Hessians $\hess_x(F^t_\mu)$ a similar bound blows up as $t\to 0^+$ \citep{chen2023logarithmic}.



\subsection{Uniform Central Limit Theorem}

We now extend the fixed-$t$ CLT to a uniform CLT for diffusion means over the time interval. The additional assumption is that $\hess_x(F^t_\mu)$ converges \emph{uniformly} on a neighborhood $U$ of the Fr\'echet mean $x^0_\star$ as $t\to 0^+$. This gives convergence of the diffusion $t$-Fr\'echet functions $F^t_\mu$ to the Fr\'echet function $F^0_\mu$ in $C^2(U)$, and allows the asymptotic behavior of diffusion means and the Fr\'echet mean to be described by a single Gaussian process. 


As $t$ varies, the tangent vectors $\rlog{x^t_\star}{x^t_n}$ lie in different tangent spaces. An intrinsic description of covariance between random vectors in different tangent spaces is developed in \cite{shao2022intrinsic}. Here, for simplicity, we fix a local orthonormal frame $E_1,\ldots,E_m$ on $U$ and express the Gaussian process and its covariance in $\bbr^m$. Under this local trivialization, we identify each $v\in T_x\calm$ with the column vector  $(\langle v,E_1(x)\rangle,\ldots,\langle v,E_m(x)\rangle)^\top$. The gradients and Hessians are also represented as vectors and matrices  in this frame.

\begin{theorem}\label{thm:clt-fre-mean}
Assume that $\mu\ll\vol_\calm$ and there is a compact, geodesically convex neighborhood $U\subseteq \calm$ of $x^0_\star$ such that $\hess_x(F^t_\mu)$ converges uniformly over $U$ as $t\to 0^+$, and that the function $\xi\mapsto \hess_x(\bms^0_\xi)$ is $\mu$-integrable for all $x\in U$. Then $F^0_\mu$ is twice differentiable on $U$ and such that
\begin{equation}\label{eq:hess-fre-multi}
    \hess_{x^0_\star}(F^0_\mu)=\lim_{t\to 0^+}\hess_{x^t_\star}(F_\mu^t)=\bbe[\hess_{x^0_\star}(\bms^0_\Xi)]+J_\mu(x^0_\star),
\end{equation}
where $J_\mu(x^0_\star)$ denotes the limit in \eqref{eq:limitJ}.
Furthermore, if $\hess_{x^0_\star}(F^0_\mu)$ is positive definite, then there exists $t_0\in (0,1]$ such that, as $n\to\infty$,
\begin{equation}\label{eq:functional-mean-clt}
 \big\{\sqn\,\rlog{x^t_\star}{x^t_n}\big\}_{t\in[0,t_0]}
 \wto \{Z_\mu^t\}_{t\in[0,t_0]}.
\end{equation}
Under a fixed orthonormal frame on $U$, $Z^t_\mu$ is a centered $\bbr^m$-valued Gaussian process, with covariance structure
\begin{equation}\label{eq:functional-mean-covariance}
\begin{aligned}
 \bbe[Z_\mu^s(Z_\mu^t)^\top]
 =\bigl(\hess^{-1}_{x^s_\star}(F^s_\mu)\bigr)\bbe\left[
  \bigl(\grad_{x^s_\star}\bms^s_\Xi\bigr)
  \bigl(\grad_{x^t_\star}\bms^t_\Xi\bigr)^\top\right]
  \bigl(\hess^{-1}_{x^t_\star}(F^t_\mu)\bigr)^{\top}.
\end{aligned}
\end{equation}
\end{theorem}


For each fixed $t\in (0,t_0]$, evaluating the Gaussian process at $t$ recovers the CLT in  \Cref{thm:clt-vf-mean}. In particular, evaluating at $t=0$ gives the CLT for Fr\'echet means,
$$
\sqrt{n}\rlog{x^0_\star}{x^0_n}\wto \caln(0,\Sigma^0),
$$
where, since $\grad_{x^0_\star}\bms^0_\Xi = -2\rlog{x^0_\star}{\Xi}$ almost surely,
$$
\Sigma^0 =  4\bigl(\hess^{-1}_{x^0_\star}(F^0_\mu)\bigr)\mathrm{Cov}\big(\rlog{x^0_\star}{\Xi}\big)
  \bigl(\hess^{-1}_{x^0_\star}(F^0_\mu)\bigr)^{\top}.
$$
The covariance formula also gives $\Sigma^t\to \Sigma^0$ as $t\to 0^+$, connecting the fixed-$t$ CLT for diffusion means with the CLT for Fr\'echet means.

If $\mu$ is supported away from the cut locus of $x^0_\star$, then $J_\mu(x^0_\star)$ vanishes identically and we have $\hess_{x^0_\star}(F^0_\mu)=\bbe[\hess_{x^0_\star}(\bms^0_\Xi)]$. This yields the familiar covariance formula for  intrinsic Fr\'echet means in \cite[Theorem 2.2]{bhattacharya2005large}. For general probability distributions, the nontrivial term $J_\mu(x^0_\star)$ in \eqref{eq:hess-fre-multi} quantifies how the geometry of the cut locus affects the statistical behavior of Fr\'echet means. Under assumptions on the density near the cut locus, \cite{hotz2024central} presented an explicit expression of $J_\mu(x^0_\star)$ and derived a general CLT for Fr\'echet means on compact Riemannian manifolds. When the assumptions of both \Cref{thm:clt-fre-mean} and \cite[Theorem 1]{hotz2024central} are satisfied, their covariance matrix agrees with $\Sigma^0$ and their additional term $J_\mu(x^0_\star)$ agrees with us up to a factor of $-1/2$ (notice that their operator $H$ is $-1/2$ times the Hessian of the squared distance function; see \cite[equation (15)]{hotz2024central}). However, our assumption on the uniform convergence of Hessians  provides a different criterion for the CLT for Fr\'echet means. In particular, it can hold even when the continuous density condition in \cite[Theorem 1(iii)]{hotz2024central} fails (see the example in \Cref{sec:counter-example}).








\section{Examples}\label{sec:example}

In the following we present explicit computations of \Cref{prop:gradient-converge} and \Cref{prop:hessian-converge} on the circle and the torus. In these two examples, we assume that the probability distribution $\mu$ is absolutely continuous to the Riemannian volume measure, with a continuous density function $\psi:\calm\to\bbr$. We show that the expressions of $J_\mu(x)$ in \eqref{eq:limitJ} coincide with the results in \cite{hotz2024central} up to a factor of $-1/2$. We further present an example on the torus with a discontinuous density that violates the conditions of \cite{hotz2024central} yet still satisfies \Cref{thm:clt-fre-mean}, thus a CLT for the Fr\'echet mean still holds.   

\subsection{Circle}

Consider the circle $\bbs^1  =\bbr/2\pi\bbz$. The heat kernel on $\bbs^1$ is given by the periodic summation of the Gauss--Weierstrass kernel \citep{nowak2019sharp}
\begin{equation}\label{eq:heat-circle}
K(t,x,y) = \frac{1}{\sqrt{2\pi t}}\sum_{n=-\infty}^\infty \exp\bigg(-\frac{(y-x+2\pi n)^2}{2t}\bigg).
\end{equation}
Therefore for $t>0$, the function $\cost^t(x,y)$ is given by 
\begin{equation*}
\cost^t(x,y) = -2t\log (K(t,x,y))=t\log(2\pi t)-2t\log\bigg(\sum_{n=-\infty}^\infty \exp\bigg(-\frac{(y-x+2\pi n)^2}{2t}\bigg)\bigg).
\end{equation*}
Without loss of generality we consider a small neighborhood $|x|\le\epsilon$ of $x=0$ and we parametrize $y=\pi+\theta$ for $\theta\in [-\pi,\pi]$. The closest point to $x$ is either $y=\pi+\theta$ or $y-2\pi=\theta-\pi$. For any $\theta\in [-\pi,\pi]$, we have the following expansion of $\cost^t(x,\theta)$:
\begin{equation}\label{eq:F-circle}
\begin{aligned}
\cost^t(x,\theta) &= t\log(2\pi t)-2t \log\bigg(\exp\bigg(-\frac{(\theta+\pi-x)^2}{2t}\bigg)+\exp\bigg(-\frac{(\theta-\pi-x)^2}{2t}\bigg)\bigg)+R^t(x,\theta)\\
&= t\log(2\pi t)-2t \log\bigg(2\exp\bigg(-\frac{(\theta-x)^2+\pi^2}{2t}\bigg)\cosh\bigg(\frac{\pi(\theta-x)}{t}\bigg)\bigg)+R^t(x,\theta)\\
&= t\log\bigg(\frac{\pi t}{2}\bigg)+(\theta-x)^2+\pi^2-2t\log\bigg(\cosh\bigg(\frac{\pi (\theta-x)}{t}\bigg)\bigg)+R^t(x,\theta).
\end{aligned}
\end{equation}
For the remainder term, we claim that
\begin{claim}\label{claim:reminder}
For any $k=0,1,2,\ldots$, 
$$
\sup_{|x|\le\epsilon,|\theta|\le\pi}|\partial_x^kR^t(x,\theta)|\to 0,\quad\text{as }t\to 0^+.
$$  
\end{claim} 
We defer the proof to \Cref{sec:proof-example}.
Taking derivative to $x$ in \eqref{eq:F-circle}, we obtain that
\begin{equation}\label{eq:grad-circle}
    \grad_x(\bms^t_\theta)= \frac{\partial}{\partial x} \cost^t(x,y) = -2(\theta-x)  +2\pi\tanh\bigg(\frac{\pi(\theta-x)}{t}\bigg)+\partial_xR^t(x,\theta).
\end{equation}
At $x=0$, by taking $t\to 0^+$, we have
\begin{equation*}
   \grad_0(\bms^t_\theta)\to \begin{cases}
     -2\theta+2\pi,&\theta>0,\\
     0,&\theta=0,\\
     -2\theta-2\pi,&\theta<0.
 \end{cases}  
\end{equation*}
Notice that for $\theta\neq 0$, $\grad_0(\bms^t_\theta)\to\grad_0(\bms^0_\theta)$. For diffusion $t$-Fr\'echet gradients, \Cref{prop:gradient-converge} tells that
\begin{equation*}
\begin{aligned}
       \grad_0(F^0_\mu) &=\lim_{t\to 0^+}\grad_0(F^t_\mu) = \bbe[\lim_{t\to 0^+}\grad_0(\bms^t_\theta)]\\
       &=-2\int_{-\pi}^\pi\theta\psi(\theta+\pi)\diff{\theta}+2\pi\int_0^\pi\psi(\theta+\pi)\diff{\theta}-2\pi\int_{-\pi}^0\psi(\theta+\pi)\diff{\theta}\\
       &= 2\pi-2\int_{0}^{2\pi}\theta\psi(\theta)\diff{\theta}+2\pi\int_\pi^{2\pi}\psi(\theta)\diff{\theta}-2\pi\int_0^{\pi}\psi(\theta)\diff{\theta}\\
       &= 4\pi\int_\pi^{2\pi}\psi(\theta)\diff{\theta}-2\int_{0}^{2\pi}\theta\psi(\theta)\diff{\theta}.
\end{aligned}
\end{equation*}

Next, taking derivative to $x$ in \eqref{eq:grad-circle}, we obtain that
\begin{equation*}
\hess_x(\bms^t_\theta) = \frac{\partial^2}{\partial x^2}\cost^t(x,y) = 2-\frac{2\pi^2}{t}\frac{1}{\cosh^2\big(\pi(\theta-x)/t\big)}+\partial^2_xR^t(x,\theta).
\end{equation*}
At $x=0$, by taking $t\to 0^+$, we have
\begin{equation*}
   \hess_0(\bms^t_\theta) \to\begin{cases}
    2,&\theta\neq 0,\\
    -\infty, &\theta=0.
\end{cases}
\end{equation*}
Notice that the Hessian of $\bms^t_\theta$ blows up near the cut locus at a rate $O(t^{-1})$.  For any $\delta>0$, by \eqref{eq:Jtdelta} we have
\begin{equation*}
J^{t,\delta}_\mu(0) = \int_{-\delta}^\delta \hess_0(\bms^t_\theta)\psi(\pi+\theta)\diff{\theta}=\int_{-\delta}^\delta\bigg(2-\frac{2\pi^2}{t\cosh^2\big(\pi\theta/t\big)}\bigg)\psi(\pi+\theta)\diff{\theta}+o_t(1).
\end{equation*}
Assuming the density function $\psi$ is continuous, by the mean value theorem, we have
\begin{equation*}
J^{t,\delta}_\mu(0) = 2\mu(\calc_\delta(0))-  2\pi\psi(\pi+\theta_t)\int_{-\delta}^\delta \frac{\pi}{t}\frac{1}{\cosh^2\big(\pi\theta/t\big)}\diff{\theta}+o_t(1),
\end{equation*}
where $\theta_t\in(-\delta,\delta)$. It follows that
\begin{equation*}
J^{t,\delta}_\mu(0) = 2\mu(\calc_\delta(0))-4\pi\psi(\pi+\theta_t)\tanh\bigg(\frac{\pi\delta}{t}\bigg)+o_t(1).
\end{equation*}
For fixed $\delta>0$, we take limit superior and limit inferior as $t\to 0^+$,
\begin{equation*}
\begin{aligned}
    &\varlimsup_{t\to 0^+}J^{t,\delta}_\mu(0)  = 2\mu(\calc_\delta(0))-4\pi\varliminf_{t\to 0^+}\psi(\pi+\theta_t)\\
    &\varliminf_{t\to 0^+}J^{t,\delta}_\mu(0)  = 2\mu(\calc_\delta(0))-4\pi\varlimsup_{t\to 0^+}\psi(\pi+\theta_t).
\end{aligned}
\end{equation*}
Then we take $\delta\to 0^+$ and we have
\begin{equation*}
J_\mu(0)=\lim_{\delta\to 0^+}\varlimsup_{t\to 0^+}J^{t,\delta}_\mu(0) = \lim_{\delta\to 0^+}\varliminf_{t\to 0^+}J^{t,\delta}_\mu(0)=-4\pi\psi(\pi).
\end{equation*}

\subsection{Torus (I)}

Consider the flat torus $\mathbb{T}^2=(\bbr/2\pi\bbz)^2$. Similar to \eqref{eq:heat-circle}, the heat kernel is given by
\begin{equation*}
K(t,\mathbf{x},\mathbf{y}) = \frac{1}{2\pi t}\sum_{\mathbf{n}\in\bbz^2}\exp\bigg(-\frac{\|\mathbf{y}-\mathbf{x}+2\pi\mathbf{n}\|^2}{2t}\bigg).
\end{equation*}
Here we also consider a small neighborhood of $\mathbf{x}=(0,0)$ and we parametrize $\mathbf{x}=(u,v)$ and $\mathbf{y} = (\pi+\theta,\pi+\omega)$ for $\theta,\omega\in [-\pi,\pi]$. We have 
\begin{equation}\label{eq:F-torus}
\begin{aligned}
&\cost^t(\mathbf{x},\mathbf{y})\\ ={} & 2t\log(2\pi t)- 2t\log\bigg(\exp\bigg(-\frac{(\theta+\pi-u)^2+(\omega+\pi-v)^2}{2t}\bigg)+\exp\bigg(-\frac{(\theta+\pi-u)^2+(\omega-\pi-v)^2}{2t}\bigg)\\
&\phantom{2tlog2pit-2tlogexp}+\exp\bigg(-\frac{(\theta-\pi-u)^2+(\omega+\pi-v)^2}{2t}\bigg)+\exp\bigg(-\frac{(\theta-\pi-u)^2+(\omega-\pi-v)^2}{2t}\bigg)\bigg)\\
& + R^t(u,\theta)+R^t(v,\omega)\\
={}& 2t\log(2\pi t)-2t \log\bigg(4\exp\bigg(-\frac{(\theta-u)^2+\pi^2}{2t}\bigg)\exp\bigg(-\frac{(\omega-v)^2+\pi^2}{2t}\bigg)\\
&\phantom{2tlogpit-2tlog4exp}\times\cosh\bigg(\frac{\pi(\theta-u)}{t}\bigg)\cosh\bigg(\frac{\pi(\omega-v)}{t}\bigg)\bigg)+ R^t(u,\theta)+R^t(v,\omega)\\
={}& 2t\log\bigg(\frac{\pi t}{2}\bigg)+2\pi^2+(\theta-u)^2-2t\log\bigg(\cosh\bigg(\frac{\pi(\theta- u)}{t}\bigg)\bigg)+(\omega-v)^2-2t\log\bigg(\cosh\bigg(\frac{\pi(\omega- v)}{t}\bigg)\bigg)\\
& + R^t(u,\theta)+R^t(v,\omega).
\end{aligned}
\end{equation}
Notice that in \eqref{eq:F-torus} the variables $u$ and $v$ are separated. The computations of gradients and Hessians are exactly the same as the case for circle. At $\mathbf{x}=(0,0)$, we simply write down the results as follows 
\begin{equation}\label{eq:gradhess-torus}
    \begin{aligned}
    &\grad_{\mathbf{0}}(\bms^t_{\theta,\omega}) = \bigg( -2\theta  +2\pi\tanh\big(\pi\theta/t\big),-2\omega  +2\pi\tanh\big(\pi\omega/t\big) \bigg)+o_t(1),\\
    &\hess_{\mathbf{0}}(\bms^t_{\theta,\omega}) \simeq \mathrm{diag}\bigg(2-\frac{2\pi^2}{t}\frac{1}{\cosh^2\big(\pi\theta/t\big)}, 2-\frac{2\pi^2}{t}\frac{1}{\cosh^2\big(\pi\omega/t\big)}\bigg)+o_t(1).
    \end{aligned}
\end{equation}
Componentwise the limits of $\grad_{\mathbf{0}}(\bms^t_{\theta,\omega})$, $\grad_{\mathbf{0}}(F^t_\mu)$ and $\hess_{\mathbf{0}}(\bms^t_{\theta,\omega})$ as $t\to 0^+$ are identical to those obtained on the circle.

For any $0<\delta<\pi$, by \eqref{eq:Jtdelta} we have
\begin{equation*}
J^{t,\delta}_\mu(\mathbf{0}) = \bigg(\int_{-\pi}^\pi\int_{-\delta}^\delta+\int_{-\delta}^\delta\int_{-\pi}^\pi-\int_{-\delta}^\delta\int_{-\delta}^\delta\bigg)\hess_{\bm{0}}(\bms^t_{\theta,\omega})\psi(\pi+\theta,\pi+\omega)\diff{\theta}\diff{\omega}.
\end{equation*}
The first diagonal element is given by
\begin{equation*}
    \begin{aligned}
J^{t,\delta}_\mu(\mathbf{0})_{11} = & 2\mu(\calc_\delta(\mathbf{0})) - \underbrace{\int_{-\delta}^\delta \bigg(\frac{2\pi^2}{t}\frac{1}{\cosh^2\big(\pi\theta/t\big)}\bigg)\bigg(\int_{-\pi}^\pi\psi(\pi+\theta,\pi+\omega)\diff{\omega}\bigg)\diff{\theta}}_{\mathrm{(I)}}\\
&-\underbrace{\int_{-\pi}^\pi \bigg(\frac{2\pi^2}{t}\frac{1}{\cosh^2\big(\pi\theta/t\big)}\bigg)\bigg(\int_{-\delta}^\delta\psi(\pi+\theta,\pi+\omega)\diff{\omega}\bigg)\diff{\theta}}_{\mathrm{(II)}}\\
&+\underbrace{\int_{-\delta}^\delta \bigg(\frac{2\pi^2}{t}\frac{1}{\cosh^2\big(\pi\theta/t\big)}\bigg)\bigg(\int_{-\delta}^\delta\psi(\pi+\theta,\pi+\omega)\diff{\omega}\bigg)\diff{\theta}}_{\mathrm{(III)}}+o_t(1).
\end{aligned}
\end{equation*}
We analyze (I)--(III) separately:
\begin{itemize}
    \item[(I)] Assuming the density function $\psi$ is continuous, by the mean value theorem, we have
\begin{equation*}
    \mathrm{(I)} = 4\pi\tanh\bigg(\frac{\pi\delta}{t}\bigg)\int_{0}^{2\pi}\psi(\pi+\theta_t,\omega)\diff{\omega},
\end{equation*}
where $\theta_t\in (-\delta,\delta)$.
\item[(II)] Notice that
\begin{equation*}
    \mathrm{(II)} \le 2\delta\|\psi\|_\infty\cdot 4\pi\tanh\bigg(\frac{\pi^2}{t}\bigg)\le 8\pi\delta\|\psi\|_\infty.
\end{equation*}
Therefore as $\delta\to 0^+$, $\mathrm{(II)}\to 0$.
\item[(III)]  Using the same argument as (II) we also have $\mathrm{(III)}\to 0$ as $\delta\to 0^+$.
\end{itemize}
In summary we have shown that
\begin{equation*}
J_\mu(\mathbf{0})_{11}=\lim_{\delta\to 0^+}\varlimsup_{t\to 0^+}J^{t,\delta}_\mu(\mathbf{0})_{11} = \lim_{\delta\to 0^+}\varliminf_{t\to 0^+}J^{t,\delta}_\mu(\mathbf{0})_{11} = -4\pi\int_{0}^{2\pi} \psi(\pi,\omega)\diff{\omega}.
\end{equation*}
The limit of the second diagonal element can be obtained in a similar way. Therefore we have
\begin{equation*}
    J_\mu(\mathbf{0})=-4\pi\begin{bmatrix}
        \displaystyle\int_0^{2\pi}\psi(\pi,\omega)\diff{\omega} & 0\\
        0 & \displaystyle\int_0^{2\pi}\psi(\theta,\pi)\diff{\theta}
    \end{bmatrix}.
\end{equation*}

\subsection{Torus (II)}\label{sec:counter-example}

On the flat torus we consider the following density function
$$
\psi(\mathbf{x}) = \psi(u,v) = f(u)f(v)\bigg(1+\frac{s(u)s(v)}{2}\bigg),
$$
where 
$$
f(u) = \frac{2+\cos(u)}{4\pi},\quad s(u)=\mathrm{sign}(\sin(u)).
$$
A direct computation shows that $f$ is also a density function on $\bbs^1$, and the marginals of $\psi$ are
$$
\int_{0}^{2\pi}\psi(u,v)\diff{u} = f(v),\quad \int_0^{2\pi}\psi(u,v)\diff{v} = f(u).
$$
The sign function causes discontinuity at $u=\pi$ and $v=\pi$. Thus the density function $\psi$ violates the continuity condition near the cut locus assumed in \cite{hotz2024central} (we will show later that $\bm{0}$ is the Fr\'echet mean). We will compute the explicit expressions of $F^t_\mu$ and show that $F^t_\mu$ converges to $F^0_\mu$ in $C^2(\mathbb{T}^2)$, hence the CLT for Fr\'echet means holds as a consequence of \Cref{thm:clt-fre-mean}.

From \eqref{eq:F-torus} we see that $\cost^t(\mathbf{x},\mathbf{y}) = \cost^t(u,\pi+\theta)+\cost^t(v,\pi+\omega)$. Therefore,
\begin{equation}\label{eq:F-t}
\begin{aligned}
    F^t_\mu(\mathbf x)
 &=\int_{-\pi}^{\pi}\int_{-\pi}^{\pi}
   \cost^t(\mathbf x,\mathbf y)
   \psi(\pi+\theta,\pi+\omega)\diff{\omega}\diff{\theta}\\
 &=\int_{-\pi}^{\pi}\cost^t(u,\pi+\theta)
      f(\pi+\theta)\diff{\theta}
   +\int_{-\pi}^{\pi}\cost^t(v,\pi+\omega)
      f(\pi+\omega)\diff{\omega}.
\end{aligned}
\end{equation}
It suffices to compute
\begin{equation}\label{eq:inte-cos}
\begin{aligned}
    \int_{-\pi}^{\pi}\cost^t(u,\pi+\theta)
      f(\pi+\theta)\diff{\theta} &= \int_{-u}^{2\pi-u}\cost^t(0,z)
      f(z+u)\diff{z}\\
      &= \frac{1}{4\pi}\int_{-\pi}^{\pi}\cost^t(0,z)(2+\cos(z+u))
      \diff{z}.
\end{aligned}
\end{equation}
To compute this integral, we use the following expression for heat kernel on the circle \citep[Lemma A.1]{eltzner2023diffusion}
$$
K(t,0,z)=\frac1{2\pi}\prod_{j=1}^{\infty}(1-e^{-jt})
 \bigl(1+2e^{-(j-1/2)t}\cos z+e^{-(2j-1)t}\bigr).
$$
Expanding the logarithm of the second factor in the infinite product, we have
\begin{align*}
 \log\bigl(1+2e^{-(j-1/2)t}\cos z+e^{-(2j-1)t}\bigr) &=2\mathbf{Re}\big\{\log\big(1+e^{-(j-1/2)t}e^{iz}\big)\big\}\\
 &=2\sum_{n=1}^{\infty}\frac{(-1)^{n+1}}{n}
       e^{-n(j-1/2)t}\cos(nz).
\end{align*}
For each fixed $t>0$, the series is absolutely convergent, and we have 
\[
 \sum_{j=1}^{\infty}e^{-n(j-1/2)t}
 =\frac{e^{-nt/2}}{1-e^{-nt}}=\frac1{2\sinh(nt/2)}.
\]
It follows that
$$
\begin{aligned}
    \cost^t(0,z) &= -2t\log (K(t,0,z))\\
    &= 2t\log(2\pi) -2t\sum_{j=1}^\infty \log(1-e^{-jt}) -4t\sum_{j=1}^\infty\sum_{n=1}^{\infty}\frac{(-1)^{n+1}}{n}
       e^{-n(j-1/2)t}\cos(nz)\\
    & = 2t\log(2\pi) -2t\sum_{j=1}^\infty \log(1-e^{-jt}) +2t\sum_{n=1}^\infty \frac{(-1)^n}{n\sinh(nt/2)}\cos(nz). 
\end{aligned}
$$
Substituting into \eqref{eq:inte-cos} and \eqref{eq:F-t}, we obtain that, for $t>0$,
$$
F^t_\mu(\mathbf{x}) = 4t\log(2\pi) -4t\sum_{j=1}^\infty \log(1-e^{-jt})-\frac{t}{2\sinh(t/2)}\big(\cos(u)+\cos(v)\big).
$$
At $t=0$, the same decomposition as \eqref{eq:F-t} holds, and we  compute that
$$
\int_{-\pi}^\pi d^2_{\bbs^1}(0,z)f(z+u)\diff{z} = \frac{1}{4\pi}\int_{-\pi}^\pi z^2(2+\cos(z+u))\diff{z}  = \frac{\pi^2}{3}-\cos(u).
$$
Thus the Fr\'echet function is given by
$$
F^0_\mu(\mathbf{x})  = \frac{2\pi^2}{3}-\big(\cos(u)+\cos(v)\big).
$$

From the expressions of $F^t_\mu$ for $t\in [0,1]$ we can see directly that $\bm{0}$ is the unique diffusion $t$-mean for every $t>0$ and the unique Fr\'echet mean at $t=0$. The Hessians are given by
$$
\hess_{\mathbf{x}}(F^t_\mu) = \frac{t}{2\sinh(t/2)}\mathrm{diag}\{\cos(u),\cos(v)\}, \quad \hess_{\mathbf{x}}(F^0_\mu) = \mathrm{diag}\{\cos(u),\cos(v)\}.
$$
Hence,
$$
\sup_{\mathbf{x}\in \mathbb{T}^2}\|\hess_{\mathbf{x}}(F^t_\mu)-\hess_{\mathbf{x}}(F^0_\mu)\|\le \bigg|\frac{t}{2\sinh(t/2)}-1\bigg|\to 0, \quad \text{as } t\to 0^+.
$$
Moreover, $\hess_{\bm{0}}(F^t_\mu)$ is positive definite for all $t\in [0,1]$. Thus the assumptions of \Cref{thm:clt-fre-mean} are satisfied and the uniform CLT holds for $t\in[0,1]$. 






\section{Discussion}

In this paper, we studied diffusion means from the perspective of Varadhan's theorem, which motivated our investigation of the geometric properties of the scaled logarithmic heat kernel on compact Riemannian manifolds and of the induced diffusion Fr\'echet functions. We analyzed the asymptotic behavior of the diffusion Fr\'echet functions, variances, and resulting diffusion means uniformly over the time interval, thereby including the Fr\'echet mean as the special case of the diffusion mean at time zero. Our results established ULLNs and UCLTs for all three objects.

Our work connects to the extensive literature on Fr\'echet means through the small time asymptotics of the gradients and Hessians of the diffusion Fr\'echet functions. In particular, we showed that the limit of the Hessian as 
$t\to 0^+$ gives rise to an integral near the cut locus. Remarkably, our results imply that this term coincides with the “non-standard” term appearing in the CLT for Fr\'echet means obtained in \cite{hotz2024central}.

\section{AI Usage Disclosure}

The author used GPT-5.6 to assist in proving \Cref{thm:clt-vf-function} and \Cref{thm:clt-fre-mean}, and in constructing the example in \Cref{sec:counter-example}. Specifically, the generative AI identified \Cref{lemma:one-sided-bound} (an equivalent form was proved in \cite{li2024matrix}) and provided a proof for \Cref{lemma:temporal}, which allowed the author to prove that certain function classes are Donsker. The original form of the density function in \Cref{sec:counter-example} was found by AI, and the presented form was modified by the author. The author has reviewed and verified all AI outputs and takes full responsibility for the final content.

\bibliographystyle{chicago}
\bibliography{references}
\newpage

\begin{appendices}

\section{Notations}

Throughout the following sections, we use $C$ to denote universal constants whose values may change from line to line but are always independent of the temporal and spatial variables. We use $\partial_t$ to denote the temporal derivative and $\partial^k_x$ to denote the $k$th covariant derivative with respect to the spatial variable $x$. We use $\lesssim$, $\gtrsim$ and $\asymp$ to denote bounds up to universal constants, and use $\simeq$ to denote asymptotic series expansions.

\section{Technical Results from Riemannian Geometry}

\subsection{Taylor Expansion of Vector Fields}



Fix a base point $x_\star\in\calm$. For any vector field $Z$, we define the following map
\begin{equation}\label{eq:root-vec-field}
\begin{aligned}
  \Pi_{x_\star}Z:\calm\setminus\mathrm{Cut}_{x_\star} &\to T_{x_\star}\calm\\
  x&\mapsto \Pi_{x\to x_\star}(Z_x),  
\end{aligned}
\end{equation}
where $\Pi_{x\to x_\star}:T_x\calm\to T_{x_\star}\calm$ is the parallel transport operator along the unique shortest geodesic from $x$ to $x_\star$. Let $\nabla$ be the Levi-Civita connection. The covariant derivative of $Z$ at $x_\star$ defines a linear map $T_{x_\star}\calm\to T_{x_\star}\calm$ by $v\mapsto \nabla_v Z$. For any smooth function $f$, its Hessian is defined as $\hess(f) = \nabla(\grad(f))$.  We have the following first order Taylor expansion of $\Pi_{x_\star}Z$.

\begin{lemma}\label{lemma:taylor-expan}
    Let $Z$ be a smooth vector field on $\calm$ and $\Pi_{x_\star}Z$ be defined as in \eqref{eq:root-vec-field}. For any $x$ in a geodesically convex neighborhood of $x_\star$, we have
    \begin{equation}\label{eq:parallel}
    (\Pi_{x_\star}Z)(x)=Z_{x_\star}+(\nabla Z)_{x_\star}(\rlog{x_\star}{x})+O(\dm^2(x,x_\star)).
    \end{equation}
    In particular, if $Z=\grad(f)$ is the gradient of a smooth function $f$, then
    $$
    (\Pi_{x_\star}\grad(f))(x) = \grad_{x_\star}(f) + \hess_{x_\star}(f)(\rlog{x_\star}{x}) + O(\dm^2(x,x_\star)).
    $$
\end{lemma}

\begin{proof}
    Denote $v=\rlog{x_\star}{x}$, hence $\gamma(s)=\rexp_{x_\star}(sv)$ is the unique shortest geodesic connecting $x_\star$ and $x$.  By the definition of covariant differentiation along a curve \citep[Corollary 4.35]{lee2018riemannian}, we have
    \begin{equation*}
    \nabla_vZ = \lim_{s\to 0}\frac{\Pi_{\gamma(s)\to x_\star}Z_{\gamma(s)}-Z_{x_\star}}{s}.
    \end{equation*}
    Since $(\Pi_{x_\star}Z)(\gamma(s))$ is a smooth curve in the vector space $T_{x_\star}\calm$, it follows that
    \begin{equation*}
    \begin{aligned}
       (\Pi_{x_\star}Z)(x)=&(\Pi_{x_\star}Z)(\gamma(1))\\
       =&(\Pi_{x_\star}Z)(\gamma(0))+\frac{\mathrm{d}}{\mathrm{d}s}\bigg|_{s=0}(\Pi_{x_\star}Z)(\gamma(s))+O(\|v\|^2)\\
       =& (\Pi_{x_\star}Z)(x_\star) + \nabla_vZ + O(\dm^2(x,x_\star)),\\
       =&Z_{x_\star}+(\nabla Z)_{x_\star}(\rlog{x_\star}{x})+O(\dm^2(x,x_\star)),
    \end{aligned}
    \end{equation*}
    where the last step follows from the fact that $(\Pi_{x_\star}Z)(x_\star)=Z_{x_\star}$ and $\nabla_v Z = (\nabla Z)_{x_\star}(v)$. 
\end{proof}

By fixing a basis of $T_{x_\star}\calm$, the map $\Pi_{x_\star}Z$ can be viewed as a map $\calm\setminus \mathrm{Cut}_{x_\star}\to\bbr^m$. Under this identification, we can regard $\rlog{x_\star}{x}$ as an $m$-dimensional column vector, and $\nabla Z$ and $\hess(f)$ as $m\times m$ matrices. Then the Taylor expansion in \Cref{lemma:taylor-expan} is simply a Taylor expansion in $\bbr^m$.

\subsection{Estimates for the Logarithmic Heat Kernel}\label{app:log-derivative}



We recall the following general bound for the heat kernel on a compact Riemannian manifold.

\begin{theorem}{\citep[Theorem 5.3.4]{hsu2002stochastic}}
Let $\calm$ be a compact Riemannian manifold. There exist positive constants $C_1$ and $C_2>0$ such that  for all $(t,x,y)\in (0,1)\times\calm\times\calm$, 
\begin{equation}\label{eq:bound-heat-kernel}
    C_1 t^{-\frac{m}{2}}e^{-\frac{\dm^2(x,y)}{2t}}\le K(t,x,y)\le C_2 t^{-\frac{2m-1}{2}}e^{-\frac{\dm^2(x,y)}{2t}}.
\end{equation}
\end{theorem}
By taking logarithms of \eqref{eq:bound-heat-kernel}, we obtain the following quantitative version of Varadhan's theorem
\begin{equation}\label{eq:quantitaive-varadhan}
    \sup_{x,y\in\calm}|\cost^t(x,y)-\cost^0(x,y)|\lesssim t(1+|\log t|)\lesssim \sqrt{t},\quad \forall t\in (0,1).
\end{equation}
We will use ``$\lesssim$'' throughout this section to denote inequalities up to universal constants independent of the spatial and temporal variables.

For derivatives of $\cost^t(x,y)$, in contrast to  Varadhan's theorem, their convergence is not uniform over the entire $\calm\times\calm$. Let $\mathrm{Cut}_\calm=\{(x,y)\in\calm\times\calm:x\in\mathrm{Cut}_y\}$. For spatial derivatives of $\cost^t(x,y)$, the following uniform convergence result away from $\mathrm{Cut}_\calm$ can be found in \cite{neel2007small}.

\begin{lemma}{\citep[Lemma 2.2]{neel2007small}}\label{lemma:derivative-convergence}
    Let $\calm$ be a compact Riemannian manifold. For each positive integer $k$, there are smooth tensor fields $b_{k,i}(x,y)$ defined on $\calm\times\calm\setminus\mathrm{Cut}_\calm$ such that for any positive integer $k$, the asymptotic expansion
    \begin{equation*}
        \partial^k_x \cost^t(x,y)\simeq \partial^k_x\cost^0(x,y) + \sum_{i=1}^\infty b_{k,i}(x,y) t^i
    \end{equation*}
    holds uniformly as $t\to 0^+$ for any compact subset of $(\calm\times\calm)\setminus\mathrm{Cut}_\calm$.
\end{lemma}

However, to prove the uniform CLT we have to use stronger  estimates for the spatial derivatives and temporal derivatives of $\cost^t(x,y)$. For spatial gradients and Hessians of $\cost^t(x,y)$, global estimates including the cut locus can be found in
\cite{chen2023logarithmic}.

\begin{lemma}{\citep[Theorem 2.1]{chen2023logarithmic}}\label{thm:log-estimate}
   Suppose $\calm$ is a compact Riemannian manifold. For any $x,y\in\calm$ and $t\in (0,1]$,
       \begin{equation}\label{eq:log-estimate}
           \begin{aligned}
               &\|\partial_x\cost^t(x,y)\|_{T_x\calm} \lesssim \sqrt{t}+\dm(x,y),\\
               &\|\partial^2_x\cost^t(x,y)\|_{T_x\calm\otimes T_x\calm} \lesssim 1+\frac{\dm^2(x,y)}{t}.
           \end{aligned}
       \end{equation}
\end{lemma}

The Hessian of $\cost^t(x,y)$ has a smallest eigenvalue that  diverges to $-\infty$ as $t\to 0^+$ at the cut locus $x\in\mathrm{Cut}_y$ \citep[Corollary 5.8]{neel2026uniform}. However, a result of \cite{li2024matrix} shows that the positive eigenvalues are uniformly bounded over $t\in (0,1]$ for all $(x,y)\in\calm\times\calm$. Thus we have the following one-sided bound.

\begin{lemma}{\citep[Theorem 1.4]{li2024matrix}}\label{lemma:one-sided-bound}
    Let $\calm$ be a compact Riemannian manifold. There is a constant $C>0$, depending only on $\calm$, such that 
    $$
    \big(\partial^2_x\cost^t(x,y)\big)[v,v]\le C\|v\|^2_{T_x\calm},
    $$
    for all $x,y\in\calm$, $v\in T_x\calm$ and $t\in(0,1]$.
\end{lemma}

For the temporal derivative of $\cost^t(x,y)$, we prove the following lemma.

\begin{lemma}\label{lemma:temporal}
There exists a constant $C>0$, depending only on $\calm$, such that 
\begin{equation}\label{eq:temporal-bound}
\sup_{x,y\in\calm} |\partial_t \cost^t(x,y)|\le \frac{C}{t}, \quad\forall t\in (0,1].
\end{equation}
In particular, for all $s,t\in [0,1]$ and $x,y\in\calm$,
\begin{equation}\label{eq:total-bound}
\sup_{\xi\in\calm}|\cost^t(x,\xi)-\cost^s(y,\xi)|\le C\big(\dm(x,y)+|t-s|^{1/3}\big).
\end{equation}
\end{lemma}

\begin{proof}
    Substituting $\cost^t(x,y) = -2t\log (K(t,x,y))$ into the heat equation, we obtain that
    $$
    \partial_t \cost^t = \frac{\cost^t}{t} +\frac{1}{2}\Delta_x \cost^t - \frac{\|\partial_x\cost^t\|^2}{4t}.
    $$
    By \Cref{thm:log-estimate} and compactness of $\calm$, we have
    $$
    \begin{aligned}
        |\partial_t \cost^t(x,y)| &\lesssim \frac{1}{t} + \frac{1}{2}\bigg(1+\frac{\dm^2(x,y)}{t}\bigg)+ \frac{\big(\sqrt{t}+\dm(x,y)\big)^2}{4t}.
    \end{aligned}
    $$
    Extracting the leading order term we see that there exists a constant $C>0$ such that \eqref{eq:temporal-bound} holds.

    To prove \eqref{eq:total-bound}, we set $0\le s< t\le 1$ and $h= t-s$, and consider the following two cases:
    \begin{itemize}
        \item[(i)] $s\le h^{2/3}$. It follows that $t\le 2h^{2/3}$. By \eqref{eq:quantitaive-varadhan},
        $$
        \begin{aligned}
            |\cost^t(x,\xi)-\cost^s(x,\xi)| &\le |\cost^t(x,\xi)-\cost^0(x,\xi)| + |\cost^s(x,\xi)-\cost^0(x,\xi)|\\
            &\lesssim \sqrt{t}+\sqrt{s} \lesssim h^{1/3}.
        \end{aligned}
        $$
        \item[(ii)] $s> h^{2/3}$. Integrating \eqref{eq:temporal-bound}, we have
        $$
        \begin{aligned}
            |\cost^t(x,\xi)-\cost^s(x,\xi)| &\lesssim \int_s^t|\partial_u\cost^u(x,\xi)|\diff{u}\\
            &\lesssim \int_s^t \frac{1}{s}\diff{u} =\frac{t-s}{s}\le h^{1/3}.
        \end{aligned}
        $$
    \end{itemize}
    In summary, we have shown that for any $x\in\calm$ and $0\le s<t\le 1$, it holds
    $$
    \sup_{\xi\in\calm}|\cost^t(x,\xi)-\cost^s(x,\xi)|\lesssim |t-s|^{1/3}.
    $$
    Using the uniform bound of $\|\partial_x\cost^t(x,y)\|$ in \eqref{eq:log-estimate}, we obtain that for any $x,y\in\calm$ and $s\in (0,1]$,
    $$
    \sup_{\xi\in\calm}|\cost^s(x,\xi)-\cost^s(y,\xi)|\lesssim \dm(x,y).
    $$
    For $s=0$, the same bound follows from $|\dm^2(x,\xi)-\dm^2(y,\xi)|\le 2 \diam(\calm)\dm(x,y)$.
    Hence by the triangle inequality we have 
    $$
    \begin{aligned}
        |\cost^t(x,\xi)-\cost^s(y,\xi)| &\le |\cost^t(x,\xi)-\cost^s(x,\xi)| + |\cost^s(x,\xi)-\cost^s(y,\xi)|\\
        &\lesssim |t-s|^{1/3}+\dm(x,y).
    \end{aligned}
    $$
    Taking the supremum over $\xi\in\calm$ we have shown \eqref{eq:total-bound} holds.
\end{proof}

\section{Background on Empirical Process Theory}

In this section we introduce some background from empirical process theory which will be needed later in the proofs of central limit theorems. We will not state the definitions and theorems in full generality. We refer to \cite{kosorok2008introduction,van1996weak,dudley2014uniform} for the full details.  

\subsection{Donsker Classes}\label{app:donsker}

Let $(\calx,\Sigma_\calx,\mu)$ be a probability space. Suppose $\xi_1,\ldots,\xi_n$ are i.i.d. samples with law $\mu$. The empirical measure is  $\mu_n=n^{-1}\sum_{i=1}^n\delta_{\xi_i}$. Let $\calf$ be any class of $\mu$-integrable real-valued functions over $\calx$. The empirical process indexed by $\calf$ is defined as 
$$
\nu_n(f) = \sqn\bigg(\int_\calx f\diff{\mu_n}-\int_\calx f\diff{\mu}\bigg)= \sqn(\mu_n-\mu)(f), \quad\forall f\in\calf.
$$
If $\sup_{f\in\calf}|\nu_n(f)|<\infty$ for all $f\in\calf$, then $\nu_n$ can be viewed as a random element in the space of bounded real-valued functions $\ell^\infty(\calf)$ equipped with the following uniform norm
$$
\|\varphi\|_\calf = \sup_{f\in\calf}|\varphi(f)|,\quad\forall\varphi\in\ell^\infty(\calf).
$$
In general $\nu_n$ may not be Borel measurable in $\ell^\infty(\calf)$. The primary goal of empirical process theory is to resolve the non-measurability problem and make sense of weak convergence of $\nu_n$ to a limiting process which is often a Gaussian process. 

Consider the space  of measurable, square-integrable functions $L^2(\mu)$. There exists a Gaussian process $G_\mu$ indexed by $L^2(\mu)$ with zero mean and covariance structure 
$$
\bbe[G_\mu(f)G_\mu(g)]=\int fg\dmu-\bigg(\int f\dmu\bigg)\bigg(\int g\dmu\bigg).
$$
The covariance structure defines a pseudometric on $L^2(\mu)$ by 
\begin{equation}\label{eq:rho-mu}
\rho_\mu(f,g) = \big(\bbe\big[\big(G_\mu(f)-G_\mu(g)\big)^2\big]\big)^{1/2}.
\end{equation}
In general the sample paths of $G_\mu$ are not continuous on the whole space $L^2(\mu)$. We are concerned with a subclass of $L^2(\mu)$ with certain sample continuity. 

\begin{definition}
    A function class $\calf\subseteq L^2(\mu)$ is called pregaussian if a Gaussian process $(f,\omega)\mapsto G_\mu(f)(\omega)$ can be defined on some probability space such that for each $\omega$, $f\mapsto G_\mu(f)(\omega)$ is bounded and uniformly continuous for $\rho_\mu$ from $\calf$ to $\bbr$. 
\end{definition}

\begin{definition}[Donsker Classes]
   A function class $\calf\subseteq L^2(\mu)$ is called a Donsker class for $\mu$, or $\mu$-Donsker class, if 
   \begin{enumerate}
       \item $\calf$ is pregaussian;
       \item $\nu_n$ converges weakly to the Gaussian process  $G_\mu$, in the sense that for any bounded continuous function $h:\ell^\infty(\calf)\to\bbr$,
       $$
       \bbe^*[h(\nu_n)]\to\bbe[h(G_\mu)],\,\text{ as }n\to\infty, 
       $$
       where $\bbe^*$ denotes the outer expectation. 
   \end{enumerate}
\end{definition}

The following characterization of Donsker classes is critical in our proof \citep[Theorem 3.7.2]{dudley2014uniform}.

\begin{theorem}\label{thm:equi-continuity}
    A function class $\calf\subseteq L^2(\mu)$ is $\mu$-Donsker if and only if
    \begin{enumerate}
        \item $\calf$ is totally bounded for $\rho_\mu$ defined as \eqref{eq:rho-mu};
        \item $\calf$ satisfies the asymptotic equicontinuity condition: For any $\epsilon>0$, 
        \begin{equation}\label{eq:equi-contiuity}
        \lim_{\delta\to 0^+}\limsup_{n\to\infty}\bbp^*\bigg[\sup_{\substack{f,g\in\calf\\\rho_\mu(f,g)<\delta}}|\nu_n(f)-\nu_n(g)|>\epsilon\bigg]=0,
        \end{equation}
        where $\bbp^*$ denotes the outer probability.
    \end{enumerate}
\end{theorem}


\subsection{Entropy Conditions}

A common criterion to determine whether a function class $\calf\subseteq L^2(\mu)$ is $\mu$-Donsker is via the \emph{bracketing entropy}. First of all we introduce the metric entropy of a metric space.

\begin{definition}
    Let $(\Theta,d_\Theta)$ be a metric space. An $\epsilon$-cover of $\Theta$ is a set $\{\theta_1,\ldots,\theta_\ell\}\subseteq \Theta$ such that for any $\theta\in\Theta$ there exists some $\theta_i$ such that $d_\Theta(\theta_i,\theta)\le \epsilon$. The $\epsilon$-covering number $N(\epsilon,\Theta,d_\Theta)$ is the minimum number of $\epsilon$-covers. The metric entropy of $\Theta$ is the logarithm of its covering number $\log N(\epsilon,\Theta,d_\Theta)$.
\end{definition}

Let $\|\cdot\|$ be any norm on the function space $\calf$. For functions there are special covering sets called \emph{brackets}.

\begin{definition}
Let $(\mathcal{F},\|\cdot\|)$ be a normed space of functions $f:\mathcal{X}\to\bbr$. Given two functions $l(\cdot)$ and $u(\cdot)$, the bracket $[l,u]$ is the set of functions $f\in\mathcal{F}$ with $l(x)\le f(x)\le u(x)$ for all $x\in\mathcal{X}$. An $\epsilon$-bracket is a bracket $[l,u]$ with $\|u-l\|\le\epsilon$. The $\epsilon$-bracketing number $N_{[\,]}(\epsilon,\mathcal{F},\|\cdot\|)$ is the minimal number of $\epsilon$-brackets covering $\mathcal{F}$. The bracketing entropy of $\calf$ is the logarithm of its bracketing number $\log N_{[\,]}(\epsilon,\calf,\|\cdot\|)$. 
\end{definition}

Since any norm $\|\cdot\|$ on $\calf$ induces a distance $d_{\|\cdot\|}$, we also have the metric entropy $\log N(\epsilon,\calf,d_{\|\cdot\|})$. A basic relation between metric entropy and bracketing entropy is given by
$$
\log N(\epsilon,\calf,d_{\|\cdot\|})\le \log N_{[\,]}(2\epsilon,\calf,\|\cdot\|).
$$
In general, the inequality cannot be reversed. However, if the function class is parametrized by a metric space $(\Theta,d_\Theta)$, under certain assumptions we can control the bracketing entropy of $\calf$ by the metric entropy of the parameter space $\Theta$.

\begin{lemma}\label{lemma:lip-entropy}
  Consider a class of functions $\calf=\{f_\theta:\theta\in\Theta\}$ of functions on $\calx$. Suppose the dependence on $\theta$ is Lipschitz in the sense that 
  $$
  \|f_{\theta_1}-f_{\theta_2}\|\le L d_\Theta(\theta_1,\theta_2),
  $$
  for some constant $L>0$. Then
  $$
  \log N_{[\,]}(2\epsilon L,\calf,\|\cdot\|)\le \log N(\epsilon,\Theta,d_\Theta)
  $$
\end{lemma}
\begin{proof}
    Let $\theta_1,\ldots,\theta_\ell$ be an $\epsilon$-cover of $\Theta$ under the metric $d_\Theta$. Then consider the brackets 
    $$
    [f_{\theta_i}-\epsilon L,f_{\theta_i}+\epsilon L].
    $$
    For any $f_\theta\in\calf$, there is some $\theta_i$ such that $d_\Theta(\theta_i,\theta)\le\epsilon$, and
    $$
    \|f_\theta-f_{\theta_i}\|\le L d_\Theta(\theta_i,\theta)\le \epsilon L.
    $$
    Thus the brackets cover $\calf$ and are of size $2\epsilon L$.
\end{proof}

A function class is Donsker if its bracketing entropy does not grow too fast. The growth rate is quantified by the so-called \emph{bracketing integral}, which is defined as
$$
J_{[\,]}(s,\calf,L_2(\mu))=\int_0^s\sqrt{\log N_{[\,]}(\epsilon,\calf,L_2(\mu))}\diff{\epsilon}.
$$

\begin{theorem}{\citep[Theorem 7.2.1]{dudley2014uniform}}\label{thm:donsker}
    Let $(\calx,\Sigma_\calx,\mu)$ be a probability space and $\calf\subseteq L^2(\mu)$. If $J_{[\,]}(1,\calf,L_2(\mu))<\infty$, then $\calf$ is $\mu$-Donsker.
\end{theorem}

\section{Proofs}\label{app:proof}

\subsection{Proofs of \Cref{sec:consistency}}
\begin{proof}[Proof of \Cref{prop:int-vf}]\,

\begin{enumerate}
    \item For any $x\in\calm$,
    $$
    |F^t_\mu(x)-F^0_\mu(x)|\le \int_\calm |\cost^t(x,\xi)-\cost^0(x,\xi)|\diff{\mu}(\xi)\le \sup_{y\in\calm}|\cost^t(x,y)-\cost^0(x,y)|.
    $$
    Taking supremum with respect to $x$ on both sides, we have
    $$
    \|F^t_\mu-F^0_\mu\|_\infty\le\sup_{x,y\in\calm}|\cost^t(x,y)-\cost^0(x,y)|.
    $$
    By Varadhan's theorem the RHS tends to 0 as $t\to 0^+$.
    \item By Lipschitz continuity of the min functional,
    $$
    |V^t_\star- V^0_\star|\le \|F^t_\mu-F^0_\mu\|_\infty\to0,\,\text{ as }t\to 0^+
    $$
    \item Assume that $\dm(x^t_\star,x^0_\star)\not\to 0$ as $t\to 0^+$. There exists $\epsilon_0>0$, a sequence $\{t_k\}$ converging to 0, and a sequence $\{x^{t_k}_\star\}\subseteq\calm$  such that  $\dm(x^{t_k}_\star,x^0_\star)>\epsilon_0$ for all $k\in\mathbb{N}$. By compactness of $\calm$ there is a convergent subsequence of $\{x^{t_k}_\star\}$. Without loss of generality we assume $x^{t_k}_\star\to x'\neq x^0_\star$, as $k\to\infty$. Notice that
    $$
    |F^0_\mu(x')-F^{t_k}_\mu(x^{t_k}_\star)|\le |F^0_\mu(x')-F^0_\mu(x^{t_k}_\star)|+|F^0_\mu(x^{t_k}_\star)-F^{t_k}_\mu(x^{t_k}_\star)|,
    $$
    where the RHS tends to 0 as $k\to\infty$ by the continuity of $F^0_\mu$ and the uniform convergence of $F^t_\mu$ to $F^0_\mu$. We then have $F^0_\mu(x')=V^0_\star=F^0_\mu(x^0_\star)$, contradicting to the uniqueness  of $x^0_\star$. Therefore the original assumption is false and $\dm(x^t_\star,x^0_\star)\to 0$ as $t\to 0^+$. 
\end{enumerate}
\end{proof}

\begin{proof}[Proof of \Cref{thm:ulln-vf-function}]
The main idea is to find, for any $\epsilon>0$, a finite sequence $t_1,\ldots,t_\ell\in[0,1]$, and $x_1,\ldots,x_\ell\in\calm$, and their neighborhoods covering the whole space, such that over each neighborhood, each term in the following is bounded by $\epsilon/3$ when $n$ is large enough.
$$
|F^t_n(x)-F^t_\mu(x)|\le \underbrace{|F^t_n(x)-F^{t_i}_n(x_i)|}_{(\mathrm{I})}+\underbrace{|F^{t_i}_n(x_i)-F^{t_i}_\mu(x_i)|}_{(\mathrm{II})}+\underbrace{|F^{t_i}_\mu(x_i)-F^t_\mu(x)|}_{(\mathrm{III})}.
$$

For any $t'\in[0,1]$ and $x'\in\calm$, consider a sequence ${\delta_k}$ of positive numbers tending to 0, and metric balls $B_k(x')=\{x\in\calm:\dm(x,x')<\delta_k\}$. Define
$$
\Delta_k(\xi)=\sup_{|t-t'|<\delta_k}\sup_{x\in B_k(x')}|\cost^t(x,\xi)-\cost^{t'}(x',\xi)|.
$$
Since $\Delta_k(\xi)$ is bounded and $\cost^t(x,\xi)$ is continuous, we have
\begin{equation}\label{eq:lim0}
\lim_{k\to\infty}\bbe[\Delta_k(\xi)]=\bbe[\lim_{k\to\infty}\Delta_k(\xi)]=0.
\end{equation}
Notice that
$$
|F^t_n(x)-F^{t'}_n(x')|\le \frac{1}{n}\sum_{i=1}^n|\cost^t(x,\xi_i)-\cost^{t'}(x',\xi_i)|,
$$
taking supremum on both sides, we obtain that
$$
\sup_{|t-t'|<\delta_k}\sup_{x\in B_k(x')}|F^t_n(x)-F^{t'}_n(x')|\le\frac{1}{n}\sum_{i=1}^n\Delta_k(\xi_i).
$$
The classical LLN implies that the RHS tends to $\bbe[\Delta_k(\Xi)]$ almost surely as $n\to\infty$. Together with \eqref{eq:lim0}, for any $\epsilon>0$, there exists $N\in\bbn$ depending on $t'$ and $x'$, such that for any $n>N$,
\begin{equation}\label{eq:sup1}
\sup_{|t-t'|<\delta_N}\sup_{x\in B_N(x')}|F^t_n(x)-F^{t'}_n(x')|<\frac{\epsilon}{3}
\end{equation}
holds almost surely. Moreover, since $F^t_\mu(x)$ is continuous with respect to $t$ and $x$, we can choose $\delta_N$ in a way such that
\begin{equation}\label{eq:continuous}
\sup_{|t-t'|<\delta_N}\sup_{x\in B_N(x')}|F^{t}_\mu(x)-F^{t'}_\mu(x')|<\frac{\epsilon}{3}
\end{equation}

For $(\mathrm{I})$ and $(\mathrm{III})$, consider the open cover given by
$$
\mathcal{U}_\epsilon=\{(t'-\delta_N,t'+\delta_N)\times B_N(x'):(t',x')\in[0,1]\times\calm\},
$$
where on each element \eqref{eq:sup1} and \eqref{eq:continuous} hold. By compactness of $[0,1]\times\calm$, there exist $t_1,\ldots,t_\ell\in[0,1]$, and $x_1,\ldots,x_\ell\in\calm$, and their neighborhoods with radii $\delta_{k_1},\ldots,\delta_{k_\ell}$, and a common $N\in\bbn$ such that for any $n>N$,
\begin{equation*}
    \sup_{|t-t_i|<\delta_{k_i}}\sup_{x\in B_{k_i}(x_i)}|F^t_n(x)-F^{t_i}_n(x_i)|<\frac{\epsilon}{3}
\end{equation*}
holds almost surely for all $i=1,\ldots,\ell$, and
$$
\sup_{|t-t_i|<\delta_{k_i}}\sup_{x\in B_{k_i}(x_i)}|F^{t}_\mu(x)-F^{t_i}_\mu(x_i)|<\frac{\epsilon}{3}
$$
holds for all $i=1,\ldots,\ell$. 

For $(\mathrm{II})$, notice that given the $t_i$'s and $x_i$'s, the quantity reduces to $\ell$ cases of classical LLN.  Increasing $N$ if necessary, we have for any $n>N$, 
$$
|F^{t_i}_n(x_i)-F^{t_i}_\mu(x_i)|<\frac{\epsilon}{3}
$$
almost surely for all $i=1,\ldots,\ell$.

In summary, we have showed that for any $\epsilon>0$, there exists $N\in\bbn$ such that for any $n>N$,
$$
\sup_{t\in[0,1]}\sup_{x\in\calm}|F^t_n(x)-F^t_\mu(x)|<\epsilon
$$
holds almost surely, and the proof is complete.
\end{proof}

\begin{proof}[Proof of \Cref{thm:ulln-vf-mean-var}]
    For the first claim, since $|V^t_n-V^t_\star|\le \|F^t_n-F^t_\mu\|_\infty$, by \Cref{thm:ulln-vf-function}, $V^t_n\asto V^t_\star$ as $n\to\infty$. For the second claim, assume that $\displaystyle\sup_{t\in[0,1]}\dm(x^t_n,x^t_\star)\centernot\asto 0$ as $n\to\infty$. With probability greater than 0, there exists $\epsilon_0>0$, a sequence $\{n_k\}$ tending to infinity, and sequences $\{x^{t_k}_{n_k}\}$ and $\{x^{t_k}_\star\}$ such that $\dm(x^{t_k}_{n_k},x^{t_k}_\star)\ge\epsilon$ for all $k$. Since $[0,1]\times\calm$ is compact, by passing to subsequences if necessary, we can assume $t_k\to t'$ and $x^{t_k}_{n_k}\to x'$ as $k\to\infty$. By \Cref{prop:int-vf}, $x^{t_k}_\star\to x^{t'}_\star$ as $k\to\infty$, hence $x'\neq x^{t'}_\star$, and $F^{t'}_\mu(x')>V^{t'}_\star$. Notice that 
    $$
    F^{t_k}_{n_k}(x^{t_k}_{n_k})-F^{t'}_\mu(x') = \big(F^{t_k}_{{n_k}}(x^{t_k}_{n_k})-F^{t_k}_\mu(x^{t_k}_{n_k})\big)+\big(F^{t_k}_\mu(x^{t_k}_{n_k})-F^{t'}_\mu(x')\big),
    $$
    where the first term tends to 0 almost surely by \Cref{thm:ulln-vf-function} and the second term tends to 0 by continuity of $F^t_\mu(x)$. Together it implies that $V^{t_k}_{n_k}\to F^{t'}_\mu(x')$ with probability greater than 0. However, since
    $$
    |V^{t_k}_{n_k}-V^{t'}_\star|\le |V^{t_k}_{n_k}-V^{t_k}_\star|+|V^{t_k}_\star-V^{t'}_\star|,
    $$
    we also have $V^{t_k}_{n_k}\asto V^{t'}_\star$ as $k\to\infty$, yielding a contradiction. Hence the assumption is false and we have $\displaystyle\sup_{t\in[0,1]}\dm(x^t_n,x^t_\star)\asto 0$ as $n\to\infty$. 
    
\end{proof}

\subsection{Proofs of \Cref{sec:clt-func-var}}
\begin{proof}[Proof of \Cref{thm:clt-vf-function}]
Consider the following function class:
\begin{equation*}
\calv=\{\bmf^t_x(\cdot)=\cost^t(x,\cdot):(t,x)\in[0,1]\times\calm\}.
\end{equation*}
    Notice that 
    $$
    \begin{aligned}
        \mu_n(\bmf^t_x)=\frac{1}{n}\sum_{i=1}^n\cost^t(x,\xi_i) = F^t_n(x),\quad
        \mu(\bmf^t_x) = \int_\calm \cost^t(x,\xi)\dmu(\xi)=F^t_\mu(x).
    \end{aligned}
    $$
    It suffices to prove that the empirical process $\nu_n=\sqn(\mu_n-\mu)$ indexed by $\calv$ converges weakly to a Gaussian process. 

    For any $s,t\in[0,1]$ and $x,y\in\calm$, we have
    $$
    \|\bmf^t_x-\bmf^s_y\|_{L^2(\mu)} = \bigg(\int_\calm |\bmf^t_x(\xi)-\bmf^s_y(\xi)|^2\dmu(\xi)\bigg)^{\frac{1}{2}}\le\|\bmf^t_x-\bmf^s_y\|_\infty.
    $$
    Thus any $\epsilon$-bracket for  $\|\cdot\|_\infty$ is also an $\epsilon$-bracket for $\|\cdot\|_{L^2(\mu)}$, and we have
    $$
    N_{[\,]}\big(\epsilon,\calv,\|\cdot\|_{L^2(\mu)}\big)\le N_{[\,]}\big(\epsilon,\calv,\|\cdot\|_\infty\big).
    $$
    By \Cref{lemma:temporal}, there exists a constant $C>1$ such that,
    $$
    \|\bmf^t_x-\bmf^s_y\|_\infty \le C\big(\dm(x,y)+|t-s|^{1/3}\big).
    $$
    As in \Cref{lemma:lip-entropy}, a metric cover of $[0,1]\times\calm$ induces a bracketing cover on the function class $\calv$. For any $0<\epsilon\le 1$, choose a uniform time grid $t_1,\ldots, t_J$ on $[0,1]$ with mesh at most $(\epsilon/(4C))^3$, so that $J\lesssim \epsilon^{-3}$. For any compact Riemannian manifold of dimension $m$, the covering number is known to be of order $N(\epsilon,\calm,\dm)\asymp \epsilon^{-m}$ \citep[Lemma B.7]{aamari2019non}. Thus there exists $x_1,\ldots,x_N$ forming an $\epsilon/(4C)$-cover of $\calm$, with $N\lesssim\epsilon^{-m}$. The brackets $[\bmf^{t_i}_{x_j}-\epsilon/2,\bmf^{t_i}_{x_j}+\epsilon/2]$ cover $\calv$ of width $\epsilon$, which implies a bound on the  bracketing number
    $$
    N_{[\,]}(\epsilon,\calv,\|\cdot\|_\infty)\lesssim  \epsilon^{-(m+3)}.
    $$
    We have the following estimate for bracketing integral,
    $$
    \begin{aligned}
        J_{[\,]}(1,\calv,L_2(\mu))=&\int_0^1\sqrt{\log N_{[\,]}(\epsilon,\calv,L^2(\mu))}\diff{\epsilon}\\
        \le&\int_0^1\sqrt{\log N_{[\,]}(\epsilon,\calv,\|\cdot\|_\infty)}\diff{\epsilon}\\
        \lesssim&\int_0^1\sqrt{1+(m+3)\log\big({\epsilon^{-1}}\big)}\diff{\epsilon}<\infty.
    \end{aligned}
    $$
    By \Cref{thm:donsker}, the  function class $\calv$ is $\mu$-Donsker. Therefore the empirical process $\nu_n$ converges weakly to a Gaussian process $G_\mu$ indexed by $\calv$. Write $G^t_\mu(x) = G_\mu(\bmf^t_x)$. This process has zero mean and covariance structure
    $$
    \begin{aligned}
        \bbe[G^t_\mu(x)G^s_\mu(y)] =& \int_\calm \bmf^t_x(\xi)\bmf^s_y(\xi)\dmu(\xi)-\bigg(\int_\calm\bmf^t_x(\xi)\dmu(\xi)\bigg)\bigg(\int_\calm\bmf^s_y(\xi)\dmu(\xi)\bigg)\\
        =&\int_\calm \bmf^t_x(\xi)\bmf^s_y(\xi)\dmu(\xi)-F^t_\mu(x)F^s_\mu(y).
    \end{aligned}
    $$
\end{proof}

\begin{proof}[Proof of \Cref{thm:clt-vf-var}]

By \Cref{prop:int-vf}, the path $t\mapsto x^t_\star$ is continuous on $[0,1]$. Consider the evaluation map defined by
$$
\begin{aligned}
    \Phi: C([0,1]\times \calm)&\to C([0,1])\\
    h &\mapsto (t\mapsto h(t,x^t_\star)).
\end{aligned}
$$
The map $\Phi$ is well-defined and satisfies
$$
\|\Phi(h)-\Phi(\widetilde{h})\|_\infty =\sup_{t\in [0,1]}|h(t,x^t_\star)-\widetilde{h}(t,x^t_\star)|\le \|h-\widetilde{h}\|_\infty.
$$
Hence by the continuous mapping theorem \citep[Theorem 7.7]{kosorok2008introduction}, as $n\to\infty$,
$$
\nu_n(\bmf^t_{x^t_\star}) = \Phi(\sqrt{n}(F_n-F_\mu))\wto \Phi(G_\mu) = G^t_\mu(x^t_\star).
$$
The covariance structure of $G^t_\mu(x^t_\star)$ follows directly from \Cref{thm:clt-vf-function}:
$$
\bbe[G^t_\mu(x^t_\star)G^s_\mu(x^s_\star)]=\bbe[\bmf^t_{x^t_\star}(\Xi)\bmf^s_{x^s_\star}(\Xi)]-V^t_\star V^s_{\star}.
$$

Consider the following decomposition
$$
\sqrt{n}(V^t_n-V^t_\star) = \nu_n(\bmf^t_{x^t_\star}) +\big(\sqrt{n}(V^t_n-V^t_\star)-\nu_n(\bmf^t_{x^t_\star})\big).
$$
We have shown that the first term converges weakly to the desired Gaussian process. If we can show further that
\begin{equation}\label{eq:target-var}
    \sup_{t\in[0,1]}\big|\sqrt{n}(V^t_n-V^t_\star)-\nu_n(\bmf^t_{x^t_\star})\big|\pto 0,
\end{equation}
then by Slutsky's theorem \citep[Theorem 7.15]{kosorok2008introduction}, we obtain the weak convergence for diffusion Fr\'echet variances.

It remains to prove \eqref{eq:target-var}. Using the definition of empirical and population means, we have 
$$
\begin{aligned}
\sqn(V^t_n-V^t_\star)
&\le\sqn\bigl(F^t_n(x^t_\star)-F^t_\mu(x^t_\star)\bigr)
=\nu_n(\bmf^t_{x^t_\star}),\\
\sqn(V^t_n-V^t_\star)
&=\nu_n(\bmf^t_{x^t_n})
  +\sqn\bigl(F^t_\mu(x^t_n)-V^t_\star\bigr)
\ge\nu_n(\bmf^t_{x^t_n}).
\end{aligned}
$$
Consequently, for every $t$ we have
\begin{equation}\label{eq:trans}
0\le\nu_n(\bmf^t_{x^t_\star})-\sqn(V^t_n-V^t_\star)
\le\nu_n(\bmf^t_{x^t_\star})-\nu_n(\bmf^t_{x^t_n}).
\end{equation}
Since the function class $\calv$ is $\mu$-Donsker, by \Cref{thm:equi-continuity}, it satisfies asymptotic equicontinuity, hence, for any $\eta>0$, 
\begin{equation}\label{eq:equi-var}
\lim_{\delta\to0^+}\limsup_{n\to\infty}
\bbp\bigg[
\sup_{\substack{t\in[0,1],\,x,y\in\calm\\ \dm(x,y)\le\delta}}
\left|\nu_n(\bmf^t_x)-\nu_n(\bmf^t_y)\right|>\eta
\bigg]=0.
\end{equation}
By \Cref{thm:ulln-vf-mean-var}, for any $\delta>0$, we have $\sup_{t\in[0,1]}\dm(x^t_n,x^t_\star)<\delta$ for sufficiently large $n$ almost surely. Thus \eqref{eq:equi-var} implies that
$$
\sup_{t\in [0,1]}\big|\nu_n(\bmf^t_{x^t_\star})-\nu_n(\bmf^t_{x^t_n})\big|\pto 0, \quad n\to\infty.
$$
Therefore, from the bound in \eqref{eq:trans}, we obtain \eqref{eq:target-var}.

\end{proof}

\subsection{Proofs of \Cref{sec:clt-mean}}

\begin{proof}[Proof of \Cref{thm:clt-vf-mean}]

Consider the tangent vector field give by
$$
\gcost^t(x,y) = \partial_x\cost^t(x,y).
$$
Let $\Pi_{x\to x^t_\star}:T_x\calm\to T_{x^t_\star}\calm$ be the parallel transport operator. For any $x\in\calm$, define 
$$
\mathbf{g}^t_x(y)=\Pi_{x\to x^t_{\star}}(\gcost^t(x,y)),
$$
By fixing a basis at $T_{x^t_\star}\calm$, we can view $\mathbf{g}^t_x$ as a map $\calm\to\bbr^m$. Let $\calg^t=\{\mathbf{g}^t_x:x\in\calm\}$ be the $\bbr^m$-valued function class. An empirical process indexed by vector-valued functions is understood as the classical scalar empirical process applied to all coordinate functions (\Cref{app:donsker}). Thus the empirical process indexed by $\calg^t$ is
$$
\begin{aligned}
    \nu_n(\mathbf{g}^t_x)=&\sqn(\mu_n-\mu)(\mathbf{g}^t_x)\\
    =&\sqn\bigg(\frac{1}{n}\sum_{i=1}^n\mathbf{g}^t_x(\xi_i)-\int_\calm \mathbf{g}^t_x(\xi)\dmu(\xi)\bigg).
\end{aligned}
$$
By \Cref{lemma:taylor-expan} we have the first order Taylor expansion for $\mathbf{g}^t_{x^t_n}$ at $x^t_\star$,
 \begin{equation}\label{eq:taylor-exp}
\mathbf{g}^t_{x^t_n}(y) = \mathbf{g}^t_{x^t_\star}(y) + \hess_{x^t_\star}(\bms^t_y)(\rlog{x^t_\star}{x^t_n}) + \rmr(x^t_n,x^t_\star,y),
\end{equation}
where for each $y\in\calm$, the remainder $\rmr(x^t_n,x^t_\star,y)=O(\dm^2(x^t_n,x^t_\star))$. Integrating over $y\in\calm$, we have
$$
\mu(\mathbf{g}^t_{x^t_n})=\mu(\mathbf{g}^t_{x^t_\star})+\hess_{x^t_\star}(F^t_\mu)(\rlog{x^t_\star}{x^t_n})+\rmr(x^t_n,x^t_\star).
$$
Since $x^t_\star$ is the unique minimizer of $F^t_\mu$, we have $\grad_{x^t_\star}(F^t_\mu)=0$, and thus
$$
\begin{aligned}
\mu(\mathbf{g}^t_{x^t_\star})&=\int_\calm\mathbf{g}^t_{x^t_\star}(\xi)\dmu(\xi)\\
&=\int_\calm\gcost^t(x^t_\star,\xi)\dmu(\xi)\\
&=\int_\calm\frac{\partial}{\partial x}\bigg|_{x=x^t_\star}\cost^t(x,\xi)\dmu(\xi)\\
&=\grad_{x^t_\star}(F^t_\mu)=0.   
\end{aligned}
$$
Similarly, at $x^t_n$ we have $\mu_n(\mathbf{g}^t_{x^t_n})=n^{-1}\sum_{i=1}^n\mathbf{g}^t_{x^t_n}(\xi_i)=0$. Therefore,
$$
\begin{aligned}
    \nu_n(\mathbf{g}^t_{x^t_n})=&\sqn(\mu_n-\mu)(\mathbf{g}^t_{x^t_n})\\
    =&-\sqn\mu(\mathbf{g}^t_{x^t_n})\\
    =&-\sqn\bigg(\hess_{x^t_\star}(F^t_\mu)(\rlog{x^t_\star}{x^t_n})+\rmr(x^t_n,x^t_\star)\bigg),
\end{aligned}
$$
which implies that
\begin{equation}\label{eq:decomp-i-iii}
\begin{aligned}
-\sqn\,\rlog{x^t_\star}{x^t_n}=&\big(\hess^{-1}_{x^t_\star}(F^t_\mu)\big)\big(\nu_n(\mathbf{g}^t_{x^t_n})+\sqn\,\rmr(x^t_n,x^t_\star)\big)\\
=&\big(\hess^{-1}_{x^t_\star}(F^t_\mu)\big)\big(\nu_n(\mathbf{g}^t_{x^t_n})+\nu_n(\mathbf{g}^t_{x^t_\star})-\nu_n(\mathbf{g}^t_{x^t_\star})+\sqn\,\rmr(x^t_n,x^t_\star)\big)\\
=&\underbrace{\big(\hess^{-1}_{x^t_\star}(F^t_\mu)\big)\big(\nu_n(\mathbf{g}^t_{x^t_\star})\big)}_{(\mathrm{I})}+\underbrace{\big(\hess^{-1}_{x^t_\star}(F^t_\mu)\big)\big(\nu_n(\mathbf{g}^t_{x^t_n})-\nu_n(\mathbf{g}^t_{x^t_\star})\big)}_{(\mathrm{II})}\\
+&\underbrace{\sqn\big(\hess^{-1}_{x^t_\star}(F^t_\mu)\big)\big(\rmr(x^t_n,x^t_\star)\big)}_{(\mathrm{III})}
\end{aligned}
\end{equation}
We analyze (I)--(III) separately in the following.
\begin{itemize}
    \item[(I)] Since $\mathbf{g}^t_{x^t_\star}(\xi_1),\ldots,\mathbf{g}^t_{x^t_\star}(\xi_n)$ are i.i.d. random vectors in $\bbr^m$, by the multidimensional CLT, we have
    $$
    \nu_n(\mathbf{g}^t_{x^t_\star})=\frac{1}{\sqn}\sum_{i=1}^n\mathbf{g}^t_{x^t_\star}(\xi_i)\wto \caln(0,\widetilde{\Sigma}), \,\text{ as }n\to\infty,
    $$
    where the covariance matrix $\widetilde{\Sigma}$ is given by
    $$
    \widetilde{\Sigma}=\bbe[\mathbf{g}^t_{x^t_\star}(\Xi)(\mathbf{g}^t_{x^t_\star}(\Xi))^\top]=\bbe\big[(\grad_{x^t_\star}(\bms^t_\Xi))(\grad_{x^t_\star}(\bms^t_\Xi))^\top\big].
    $$
    Consequently, we have 
    $$
    \big(\hess^{-1}_{x^t_\star}(F^t_\mu)\big)(\nu_n(\mathbf{g}^t_{x^t_\star}))\wto \caln(0,\Sigma^t), \,\text{ as }n\to\infty.
    $$
    where $\Sigma^t = \big(\hess^{-1}_{x^t_\star}(F^t_\mu)\big)\widetilde{\Sigma}\big((\hess^{-1}_{x^t_\star}(F^t_\mu)\big)^\top$.
    \item[(II)] We would like to show $\|\nu_n(\mathbf{g}^t_{x^t_n})-\nu_n(\mathbf{g}^t_{x^t_\star})\|_{T_{x^t_\star}\calm}\pto 0$ as $n\to\infty$. For any $k=1,\ldots,m$, let $(\mathbf{g}^t_x)_k$ be the $k$-th component of $\mathbf{g}^t_x$. It suffices to show that
    \begin{equation}\label{eq:nu-g-k-continuous}
    |\nu_n((\mathbf{g}^t_{x^t_n})_k)-\nu_n((\mathbf{g}^t_{x^t_\star})_k)|\pto 0,\,\text{ as }n\to\infty.
    \end{equation}
    Consider the following function class
$$
\calg^t_k=\{(\mathbf{g}^t_x)_k:x\in\calm\}.
$$
Since $\cost^t$ is smooth for any $t>0$, the function $(\mathbf{g}^t_x)_k$ is also smooth, and 
$$
\begin{aligned}
    \|(\mathbf{g}^t_x)_k-(\mathbf{g}^t_y)_k\|_\infty&\le \bigg(\sup_{z,w\in\calm}\big\|\grad_z\big((\mathbf{g}^t_w)_k\big)\big\|_{T_z\calm}\bigg)\dm(x,y)\le L_k\dm(x,y).
\end{aligned}
$$
Hence $\calg^t_k$ is Lipschitz with respect to the spatial parameter $x$, by \Cref{lemma:lip-entropy}, each $\calg_k$ is a $\mu$-Donsker class. By \Cref{thm:equi-continuity}, the function class $\calg^t_k$ is asymptotically equicontinuous, meaning that for any $\epsilon,\eta>0$, there exists $\delta>0$ such that for sufficiently large $n$, 
        \begin{equation}\label{eq:nu-g-k-pr}
        \bbp\big[\textstyle\sup_{\rho_\mu((\mathbf{g}^t_x)_k,(\mathbf{g}^t_y)_k)<\delta}\big\{|\nu_n((\mathbf{g}^t_x)_k)-\nu_n((\mathbf{g}^t_y)_k)|\big\}>\eta\big]<\epsilon.        \end{equation}
        Notice that
        $$
        \begin{aligned}
            \rho_\mu((\mathbf{g}^t_x)_k,(\mathbf{g}^t_y)_k)\le&\bigg(\int_\calm\big((\mathbf{g}^t_x)_k(\xi)-(\mathbf{g}^t_y)_k(\xi)\big)^2\dmu(\xi)\bigg)^{\frac{1}{2}}\\ 
            \le&\|(\mathbf{g}^t_x)_k-(\mathbf{g}^t_y)_k\|_\infty 
            \le L_k\dm(x,y).
        \end{aligned}
        $$
        For each $n$, define the sets
        $$
        \begin{aligned}
            &A_n=\{\rho_\mu((\mathbf{g}^t_{x^t_n})_k,(\mathbf{g}^t_{x^t_\star})_k)<\delta\},\\
            &B_n=\big\{\textstyle\sup_{\rho_\mu((\mathbf{g}^t_x)_k,(\mathbf{g}^t_y)_k)<\delta}\big\{|\nu_n((\mathbf{g}^t_x)_k)-\nu_n((\mathbf{g}^t_y)_k)|\big\}\le\eta\big\}.
        \end{aligned}
        $$
        By \Cref{thm:ulln-vf-mean-var}, $\dm(x^t_n,x^t_\star)\asto 0$, hence for sufficiently large $n$, $\bbp[A_n]\ge1-\epsilon$. By \eqref{eq:nu-g-k-pr}, we have $\bbp[B_n]\ge 1-\epsilon$ for sufficiently large $n$. Therefore,
        $$
        \bbp[\{|\nu_n((\mathbf{g}^t_{x^t_n})_k)-\nu_n((\mathbf{g}^t_{x^t_\star})_k)|\le\eta\}]\ge\bbp[A_n\cap B_n]\ge 1-2\epsilon.
        $$
        We obtain that for sufficiently large $n$,
        $$
        \bbp[\{|\nu_n((\mathbf{g}^t_{x^t_n})_k)-\nu_n((\mathbf{g}^t_{x^t_\star})_k)|>\eta\}]\le 2\epsilon,
        $$
        which is equivalent to \eqref{eq:nu-g-k-continuous}.
    \item[(III)] Since $\rmr(x^t_n,x^t_\star)=o_\bbp(\dm(x^t_n,x^t_\star))=o_\bbp(\|\rlog{x^t_\star}{x^t_n}\|)$, it follows that $(\mathrm{III})=o_\bbp(\sqn\,\|\rlog{x^t_\star}{x^t_n}\|)$. From the previous analyses of (I) and (II), we have
    $$
    \sqn\,\|\rlog{x^t_\star}{x^t_n}\| = O_\bbp(1) + o_\bbp(1) + o_\bbp(\sqn\,\|\rlog{x^t_\star}{x^t_n}\|).
    $$
    Hence $\sqn\,\rlog{x^t_\star}{x^t_n}=O_\bbp(1)$ and $(\mathrm{III})=o_\bbp(1)$.
\end{itemize}

In summary, we have proved that $(\mathrm{II})+(\mathrm{III})\pto 0$, and thus $\sqn\,\rlog{x^t_\star}{x^t_n}\wto \caln(0,\Sigma^t)$,
where $\Sigma^t$ is given as in (I).

\end{proof}

\begin{proof}[Proof of \Cref{prop:gradient-converge}]
    By \Cref{thm:log-estimate} (ii), for any $x\notin \mathrm{Cut}_\xi$, we have $\grad_x(\bms^t_\xi)\to\grad_x(\bms^0_\xi)$. Since $x\in\mathrm{Cut}_\xi$ if and only if $\xi\in\mathrm{Cut}_x$, it follows that $\grad_x(\bms^t_\xi)\to\grad_x(\bms^0_\xi)$ for all $\xi\notin \mathrm{Cut}_x$. By \Cref{thm:log-estimate} (i), the norm of $\grad_x(\bms^t_\xi)$ is uniformly bounded, therefore by Lebesgue's dominated convergence theorem, we have
    \begin{align}
        \lim_{t\to 0^+}\grad_x(F^t_\mu) &= \lim_{t\to 0^+}\int_\calm \grad_x(\bms^t_\xi)\diff{\mu}(\xi)\tag{by smoothness of $\cost^t(x,y)$ }\\
        &= \lim_{t\to 0^+}\int_{\calm\setminus\mathrm{Cut}_x} \grad_x(\bms^t_\xi)\diff{\mu}(\xi)\tag{by assumption $\mu\ll\vol_\calm$}\\
        &= \int_{\calm\setminus\mathrm{Cut}_x} \lim_{t\to 0^+}\grad_x(\bms^t_\xi)\diff{\mu}(\xi)\tag{by dominated convergence}\\
        &= \int_{\calm\setminus\mathrm{Cut}_x} \grad_x(\bms^0_\xi)\diff{\mu}(\xi)\tag{by \Cref{lemma:derivative-convergence}}\\
        &= \grad_x(F^0_\mu) \notag,  
    \end{align}
    which proves the claim.
\end{proof}

\begin{proof}[Proof of \Cref{prop:hessian-converge}]
   On the compact set $\calm\setminus\calc_\delta(x)\subseteq\calm\setminus\mathrm{Cut}_x$, by \Cref{lemma:derivative-convergence}, we have uniform convergence $\hess_x(\bms^t_\xi)\to\hess_x(\bms^0_\xi)$ as $t\to 0^+$. Therefore,
   \begin{equation*}
       \int_{\calm\setminus\calc_\delta(x)}\hess_x(\bms^t_\xi)\diff{\mu}(\xi)\to \int_{\calm\setminus\calc_\delta(x)}\hess_x(\bms^0_\xi)\diff{\mu}(\xi), \text{ as } t\to 0^+.
   \end{equation*}
   It follows that
   \begin{equation}\label{eq:lim-t}
       \begin{aligned}
           & \varlimsup_{t\to 0^+}\hess_x(F^t_\mu) = \int_{\calm\setminus\calc_\delta(x)}\hess_x(\bms^0_\xi)\diff{\mu}(\xi)+ \varlimsup_{t\to 0^+} J_\mu^{t,\delta}(x),\\
           & \varliminf_{t\to 0^+}\hess_x(F^t_\mu) = \int_{\calm\setminus\calc_\delta(x)}\hess_x(\bms^0_\xi)\diff{\mu}(\xi)+ \varliminf_{t\to 0^+} J_\mu^{t,\delta}(x).
       \end{aligned}
   \end{equation}
   Since $\xi\mapsto \hess_x(\bms^0_\xi)$ is $\mu$-integrable, by taking $\delta\to 0^+$ on both sides of \eqref{eq:lim-t}, we have
      \begin{equation}\label{eq:lim-delta}
       \begin{aligned}
           & \varlimsup_{t\to 0^+}\hess_x(F^t_\mu) = \bbe[\hess_x(\bms^0_\Xi)]+ \lim_{\delta\to 0^+}\varlimsup_{t\to 0^+} J_\mu^{t,\delta}(x),\\
           & \varliminf_{t\to 0^+}\hess_x(F^t_\mu) = \bbe[\hess_x(\bms^0_\Xi)]+ \lim_{\delta\to 0^+}\varliminf_{t\to 0^+} J_\mu^{t,\delta}(x).
       \end{aligned}
   \end{equation}
   Hence $\displaystyle\lim_{t\to 0^+}\hess(F^t_\mu)$ exists if and only if \eqref{eq:limitJ} holds.
\end{proof}

\begin{proof}[Proof of \Cref{thm:clt-fre-mean}]
For any $x\in U$, let $\gamma:[0,1]\to\calm$ be the unique geodesic connecting $x^0_\star$ and $x$, we have
    \begin{equation}\label{eq:int-grad}
    \begin{aligned}
       \Pi_{x\to x^0_\star}\big(\grad_x(F^t_\mu)\big)=&\Pi_{\gamma(1)\to x^0_\star}\big(\grad_x(F^t_\mu)\big)\\
       =&\grad_{x^0_\star}(F^t_\mu)+\int_0^1\frac{\mathrm{d}}{\mathrm{d}s}\bigg(\Pi_{\gamma(s)\to x^0_\star}\big(\grad_{\gamma(s)}(F^t_\mu)\big)\bigg)\diff{s}\\
       =&\grad_{x^0_\star}(F^t_\mu)+\int_0^1\Pi_{\gamma(s)\to x^0_\star}\big(\hess_{\gamma(s)}(F^t_\mu)\gamma'(s)\big)\diff{s}.
    \end{aligned}
    \end{equation}
By the assumption that $\hess_x(F^t_\mu)$ converges uniformly over $U$ as $t\to 0^+$, \eqref{eq:int-grad} and \Cref{prop:gradient-converge} imply that $\grad_x(F^t_\mu)$ also converges uniformly over $U$ as $t\to 0^+$. Together with \Cref{prop:int-vf} (i), it implies that $F^t_\mu$ converges in $C^2(U)$-norm to $F^0_\mu$ as $t\to 0^+$. Since $C^2(U)$ is a Banach space, it follows that $F^0_\mu\in C^2(U)$ and $\hess_x(F^0_\mu)=\displaystyle\lim_{t\to 0^+}\hess_x(F^t_\mu)$ for any $x\in U$. In particular, by \Cref{prop:hessian-converge}, we have
\begin{equation*}
     \hess_{x^0_\star}(F^0_\mu)=\bbe[\hess_{x^0_\star}(\bms^0_\Xi)]+J_\mu(x^0_\star).
\end{equation*}
By \Cref{prop:int-vf} (iii) $x^t_\star\to x^0_\star$ as $t\to 0^+$, thus we also have $\hess_{x^0_\star}(F^0_\mu)=\displaystyle\lim_{t\to 0^+}\hess_{x^t_\star}(F^t_\mu)$.

The uniform convergence of Hessians also imply that the Taylor expansion in \eqref{eq:taylor-exp} holds for all $t\in [0,1]$. When $\hess_{x^0_\star}(F^0_\mu)$ is positive definite, there exists $t_0\in (0,1]$ such that, by shrinking $U$ if necessary,  $\hess_{x}(F^t_\mu)$ is positive definite for all $(t,x)\in [0,t_0]\times U$, and $x^t_\star\in U$ for all $t\in [0,t_0]$. Under a  fixed smooth orthonormal frame, we identify $\mathbf{g}^t_x(\xi) = \Pi_{x\to x^t_\star}\grad_x\bms^t_\xi$ as an element in $\bbr^m$, moreover, the decomposition  \eqref{eq:decomp-i-iii} into three terms (I)--(III) still holds. However, we need to consider the  $\bbr^m$-valued function class involving the time parameter
$$
\calg = \{\mathbf{g}^t_x(\cdot):(t,x)\in [0,t_0]\times U\}.
$$
If we can show $\calg$ is Donsker, then for each term in the decomposition \eqref{eq:decomp-i-iii}, we can show that
\begin{itemize}
    \item[(I)] The empirical process $\nu_n$ weakly converges to a centered $\bbr^m$-valued Gaussian process $\widetilde{Z}_\mu$ indexed by $\calg$. Evaluating $\widetilde{Z}_\mu$ on $\mathbf{g}^t_{x^t_\star}$, we obtain the desired Gaussian process indexed by $[0,t_0]$,
    $$
    Z^t_\mu = -(\hess^{-1}_{x^t_\star}(F^t_\mu))\widetilde{Z}_\mu(\mathbf{g}^t_{x^t_\star}).
    $$
    The covariance structure is given by
    $$
    \begin{aligned}
 \bbe[Z_\mu^s(Z_\mu^t)^\top]
 =\hess_{x^s_\star}^{-1}(F^s_\mu)\bbe\big[
       \mathbf g_{x^s_\star}^s(\Xi)
       \bigl(\mathbf g_{x^t_\star}^t(\Xi)\bigr)^\top\big](\hess_{x^t_\star}^{-1}(F^t_\mu))^\top;
\end{aligned}
    $$
    \item[(II)] By \Cref{thm:ulln-vf-mean-var} we have $x^t_n\asto x^t_\star$ uniformly over $t\in [0,t_0]$. The asymptotic equicontinuity implies that, as $n\to\infty$,
    $$
    \sup_{t\in [0,t_0]}\|\nu_n(\mathbf{g}^t_{x^t_n})-\nu_n(\mathbf{g}^t_{x^t_\star})\|\pto 0.
    $$
    Thus (II) is $o_{\mathbb{P}}(1)$ uniformly over $t\in [0,t_0]$;
    \item[(III)] Using the same argument as in \Cref{thm:clt-vf-mean}, the remainder term is $o_{\bbp}(1)$ uniform over $t\in [0,t_0]$.
\end{itemize}
Altogether we obtain that $\{\sqrt{n}\rlog{x^t_\star}{x^t_n}\}_{t\in [0,t_0]}\wto \{Z^t_\mu\}_{t\in [0,t_0]}$ as $n\to \infty$.

It remains to prove that $\calg$ is Donsker, i.e., for any $1\le k\le m$, the $k$-th coordinate function class $\calg_k=\{(\mathbf{g}^t_x)_k:(t,x)\in [0,t_0]\times U\}$ is Donsker. Let $e_k = (0,\ldots,1,\ldots,0)$, $1\le k\le m$ be unit coordinate vectors in the chart containing $U$. By \Cref{lemma:one-sided-bound}, there exists a constant $C>0$ such that 
\begin{equation}\label{eq:sec-der-bound}
\partial_{kk}\bms^t_{\xi}(x)\le C,\quad \forall x\in U. 
\end{equation}
For any $t\in (0,t_0]$ and $x\in U$, consider Taylor expansions of $\bms^t_\xi$ in the directions of $e_k$ and $-e_k$:
$$
\begin{aligned}
    &\bms^t_\xi(x+ \delta e_k) = \bms^t_\xi(x) +\delta\partial_k\bms^t_\xi(x) +\frac{\delta^2}{2}\partial_{kk}\bms^t_{\xi}(x_{\delta})\\
    &\bms^t_\xi(x- \delta e_k) = \bms^t_\xi(x) -\delta\partial_k\bms^t_\xi(x) +\frac{\delta^2}{2}\partial_{kk}\bms^t_{\xi}(x'_{\delta}).
\end{aligned}
$$
By \eqref{eq:sec-der-bound}, we get 
\begin{equation}\label{eq:sandwich-j-deri}
l_{t,x,\delta}(\xi):=\frac{\bms^t_\xi(x+ \delta e_k)-\bms^t_\xi(x)}{\delta}-\frac{C\delta}{2}\le \partial_k\bms^t_{\xi}(x)\le \frac{\bms^t_\xi(x)-\bms^t_\xi(x- \delta e_k)}{\delta}+\frac{C\delta}{2} =: u_{t,x,\delta}(\xi).
\end{equation}
Define $q^t(\delta) = \bms^t_\xi(x+\delta e_k)-C\delta^2/2$. By \eqref{eq:sec-der-bound}, $q^t(\delta)$ is a concave function. Since $\bms^t_\xi\to \bms^0_\xi$ uniformly by Varadhan's theorem, it follows that $q^0(\delta) = \bms^0_\xi(x+\delta e_k)-C\delta^2/2$ is also a concave function. By concavity $q^0(\delta)\le q^0(0)+(q^0)'(0)\delta$, we have \eqref{eq:sandwich-j-deri} at $t=0$.

By \Cref{thm:log-estimate}, there is $L>0$ such that $|\partial_k \bms^t_\xi|\le L$ for all $t\in [0,t_0]$. Thus we have
$$
0\le u_{t,x,\delta}-l_{t,x,\delta} = C\delta+\frac{2\bms^t_\xi(x)-\bms^t_\xi(x-\delta e_k)-\bms^t_\xi(x+\delta e_k)}{\delta} \le C\delta + 2L.
$$
From the assumption that $\hess_x(F^t_\mu)$ is positive definite for $(t,x)\in [0,t_0]\times U$, there exists $B>0$ such that $|\partial_{kk}F^t_\mu|\le B$. We obtain that
$$
0\le \int_{\calm}(u_{t,x,\delta}(\xi)-l_{t,x,\delta}(\xi))\diff{\mu}(\xi) = C\delta+\frac{2F^t_\mu(x)-F^t_\mu(x-\delta e_k)-F^t_\mu(x+\delta e_k)}{\delta} \le C\delta+B\delta.
$$
It follows that 
$$
\|u_{t,x,\delta}-l_{t,x,\delta}\|^2_{L^2(\mu)}\le (C\delta+2L)(C+B)\delta\lesssim \delta.
$$
Hence $[l_{t,x,\delta},u_{t,x,\delta}]$ gives a bracket of width $O(\sqrt{\delta})$ under $L^2(\mu)$. Notice that by \Cref{lemma:temporal}, for any $t,s\in[0,t_0]$ and $x,y\in U$, there exists $C'>0$ such that
$$
\begin{aligned}
    |u_{t,x,\delta}(\xi)-u_{s,y,\delta}(\xi)| &\le \frac{1}{\delta}\big(|\bms^t_\xi(x-\delta e_k)-\bms^s_\xi(y-\delta e_k)|+|\bms^t_\xi(x)-\bms^s_\xi(y)|\big)\\
    &\le \frac{C'}{\delta}\big(\dm(x,y)+|t-s|^{1/3}\big),
\end{aligned}
$$
and the same bound holds for $l_{t,x,\delta}$. Now for any sufficiently small $\epsilon>0$, setting $\delta=\epsilon^2$, choose a set $\{t_1,\ldots, t_J\}$ which forms a $(\delta\epsilon/2C')^3$-cover of $[0,t_0]$ and a set $\{x_1,\ldots, x_N\}$ which forms a $(\delta\epsilon/2C')$-cover of $U$. The set $\{[l_{t_j,x_i,\delta}-\epsilon,u_{t_j,x_i,\delta}+\epsilon]:1\le j\le J, 1\le i\le N\}$ has brackets of width $O(\epsilon)$ under $L^2(\mu)$ and the total number is
$$
NJ\lesssim (\delta\epsilon)^{-3}(\delta\epsilon)^{-m} = \epsilon^{-(9+3m)}.
$$
Since parallel transport is an isometry between tangent spaces, the above estimate implies that
$$
\int_0^1\sqrt{\log N_{[]}(\epsilon,\calg_k,L^2(\mu))\diff{\epsilon}}\lesssim \int_0^1\sqrt{(9+3m)\log (\epsilon^{-1})}\diff{\epsilon}<\infty.
$$
By \Cref{thm:donsker}, the class $\calg_k$ is Donsker for each $1\le k\le m$. Hence we complete the  proof.
\end{proof}

\subsection{Proofs of \Cref{sec:example}}\label{sec:proof-example}

\begin{proof}[Proof of \Cref{claim:reminder}]
Truncating the series expansion of $K(t,x,\pi+\theta)$ into two parts, we define
\begin{equation*}
\begin{aligned}
S^t(x,\theta)
&=\exp\bigg(-\frac{(\theta+\pi-x)^2}{2t}\bigg)
 +\exp\bigg(-\frac{(\theta-\pi-x)^2}{2t}\bigg),\\
Q^t(x,\theta)
&=\sum_{n=1}^{\infty}\bigg[
 \exp\bigg(-\frac{(\theta+\pi-x+2\pi n)^2}{2t}\bigg)
 +\exp\bigg(-\frac{(\theta-\pi-x-2\pi n)^2}{2t}\bigg)\bigg].
\end{aligned}
\end{equation*}
The remainder term $R^t$ in \eqref{eq:F-circle} is 
\begin{equation*}\label{eq:circle-remainder-log}
R^t(x,\theta)=-2t\log\bigg(1+\frac{Q^t(x,\theta)}{S^t(x,\theta)}\bigg).
\end{equation*}
For each integer $k\ge0$ and any $a\in\bbr$, repeated differentiation gives the following fundamental estimate
\begin{equation}\label{eq:diff-exp}
\bigg|\partial_x^k\exp\bigg(-\frac{(a-x)^2}{2t}\bigg)\bigg|
\le C_k t^{-k}(1+|a-x|)^k\exp\bigg(-\frac{(a-x)^2}{2t}\bigg),
\end{equation}
For simplicity we will use $C_k$ to denote constants only depending on $k$ (the values may vary).  Since $|x|\le\epsilon, |\theta|\le \pi$, applying \eqref{eq:diff-exp} to $S^t$, we have
\begin{equation*}
    |\partial_x^kS^t|\le C_k t^{-k}S^t.
\end{equation*}
Notice that for any $n\ge1$,
\begin{equation*}
\begin{aligned}
(\theta+\pi-x+2\pi n)^2-(\theta+\pi-x)^2
 &=4\pi n(\theta+\pi-x+\pi n)\ge4\pi n(\pi n-\epsilon),\\
(\theta-\pi-x-2\pi n)^2-(\theta-\pi-x)^2
 &=4\pi n(\pi n-\theta+\pi+x)\ge4\pi n(\pi n-\epsilon).
\end{aligned}
\end{equation*}
Applying \eqref{eq:diff-exp} to $Q^t$, we have
\begin{equation*}
|\partial_x^kQ^t|
\le C_k t^{-k}S^t\sum_{n=1}^{\infty}(1+n)^k
\exp\bigg(-\frac{2\pi n(\pi n-\epsilon)}{t}\bigg).
\end{equation*}
For the last sum, noting that $\pi n^2-\epsilon n=\pi n(n-1)+(\pi-\epsilon)n$, we obtain that 
\begin{equation*}
\begin{aligned}
\sum_{n=1}^{\infty}(1+n)^k\exp\bigg(-\frac{2\pi n(\pi n-\epsilon)}{t}\bigg)
&\le \exp\bigg(-\frac{2\pi(\pi-\epsilon)}{t}\bigg)
\sum_{n=1}^{\infty}(1+n)^k e^{-2\pi^2n(n-1)}\\
&\le C_k\exp\bigg(-\frac{2\pi(\pi-\epsilon)}{t}\bigg).
\end{aligned}
\end{equation*}
Consequently,
\begin{equation*}\label{eq:circle-omitted-bound}
|\partial_x^kQ^t|\le C_k t^{-k}S^t
\exp\bigg(-\frac{2\pi(\pi-\epsilon)}{t}\bigg).
\end{equation*}
Using the above estimates for $\partial^k_x S^t$ and $\partial^k_x Q^t$, we can obtain the following estimates through  Leibniz's rule and induction on $k$.
\begin{itemize}
    \item Differentiating $S^t(S^t)^{-1}=1$ gives,
    $$
    \big|\partial_x^k(S^t)^{-1}\big|\le C_k t^{-k}(S^t)^{-1}.
    $$
    \item Differentiating $Q^t(S^t)^{-1}$ gives, 
    \begin{equation}\label{eq:diff-quotient}
\big|\partial_x^k\big(Q^t/S^t\big)\big|
\le C_k t^{-k}\exp\bigg(-\frac{2\pi(\pi-\epsilon)}{t}\bigg).
    \end{equation}
\end{itemize}
Therefore, for $t$ small enough so that $C_ke^{-2\pi(\pi-\epsilon)/t}<1/2$ for the $C_k$ in \eqref{eq:diff-quotient}, we have
$$
\begin{aligned}
   \big|\partial^k_x\log\big(1+Q^t/S^t\big)\big| &\le \sum_{j=1}^\infty \frac{1}{j} \bigg|\partial^k_x\big(\big(Q^t/S^t\big)^j\big)\bigg|\\
   & \le \sum_{j=1}^\infty \frac{1}{j} \sum_{\ell_1+\cdots+\ell_j=k}\frac{k!}{\ell_1!\cdots\ell_j!}\prod_{i=1}^j\big|\partial^{\ell_i}_x\big(Q^t/S^t\big)\big|\\
   &\le \sum_{j=1}^\infty \frac{1}{j}\sum_{\ell_1+\cdots+\ell_j=k}\frac{k!}{\ell_1!\cdots\ell_j!}\bigg(t^{-k}\big(C_ke^{-2\pi(\pi-\epsilon)/t}\big)^j\bigg)\\
   &=t^{-k}\sum_{j=1}^\infty j^{k-1}\big(C_ke^{-2\pi(\pi-\epsilon)/t}\big)^j\\
   &\le C_k t^{-k} e^{-2\pi(\pi-\epsilon)/t}\bigg(1+\sum_{j=2}^\infty j^{k-1}2^{-(j-1)} \bigg).
\end{aligned}
$$
It follows that the remainder term satisfies
$$
\sup_{|x|\le\epsilon,|\theta|\le \pi}\big|\partial^k_x R^t(x,\theta)\big|\le C_k t^{1-k}\exp\bigg(-\frac{2\pi(\pi-\epsilon)}{t}\bigg)\to 0,\quad \text{as }t\to 0^+, 
$$
which proves the claim.

\end{proof}

\end{appendices}

\end{document}